\documentclass[12pt]{article}

\usepackage{amsfonts,amsthm,mathrsfs,amssymb,amsmath,float,tikz}
\usepackage{enumitem}
\usetikzlibrary{arrows}
\usepackage{caption}

\usepackage{tikz}
\usepackage[all]{xy}

\usetikzlibrary{arrows}
\usetikzlibrary{shapes.arrows}
\usepackage{hyperref}
    
\usepackage{color,hyperref}
    \catcode`\_=11\relax
    \newcommand\email[1]{\_email #1\q_nil}
    \def\_email#1@#2\q_nil{%
      \href{mailto:#1@#2}{{\emailfont #1\emailampersat #2}}
    }
    \newcommand\emailfont{\sffamily}
    \newcommand\emailampersat{{\color{red}\small@}}
    \catcode`\_=8\relax

\theoremstyle{plain}
\newtheorem{thm}{Theorem}[section]
\newtheorem*{thm*}{Theorem}
\newtheorem{lemma}[thm]{Lemma}
\newtheorem{cor}[thm]{Corollary}

\newcommand{\be}{\begin{equation}} 
\newcommand{\ee}{\end{equation}}

\newcommand{\sa}{\Sigma}
\newcommand{\M}{\mathbf{Meas}}
\newcommand{\Mc}{\iota}
\newcommand{\C}{\mathbf{Cvx}}
\newcommand{\B}{\mathcal{B}}
\newcommand{\G}{\mathcal{G}}   

\newcommand{\T}{\mathcal{P}}   
\newcommand{\U}{\mathcal{U}}
\newcommand{\F}{\mathcal{F}}

\newcommand{\mS}{\mathcal{S}}
\newcommand{\I}{I}    
\newcommand{\Rbar}{I}                         

\newcommand{\Ri}{\mathcal{R}^{\iota}}
\newcommand{\Xit}{(X \! \! \downarrow \! \! \iota)}
\newcommand{\Set}{\mathbf{Set}}

\author{Kirk Sturtz}

\title{Categorical Probability Theory}

\begin{document}
 \maketitle

\begin{abstract}     We present a categorical viewpoint of probability measures by showing that a probability measure  can  be  viewed as a weakly averaging affine measurable functional taking values in the unit interval  which preserves limits.    The probability measures on a space are the elements of a submonad  of a double dualization monad on the category of measurable spaces into the unit interval,  and this monad is naturally isomorphic to the Giry monad.  We show this submonad is the codensity monad of a functor from the category of convex spaces to the category of measurable spaces.   A theorem proving the integral operator acting on the space of measurable functions and the space of probability measures on the domain space of those functions is given using the strong monad structure of the Giry monad.
\end{abstract}
 
 \section{Introduction}  

In the paper \emph{Codensity and the ultrafilter monad} \cite{Leinster} discusses how the ultrafilter monad $\U$ on the category of sets, $\Set$, is a subfunctor of the double dualization monad $\Set( 2^{\bullet}, 2)$ where we use the exponent notation, $2^X = \Set(X,2)$ for every object $X$ in $\Set$, for brevity.   In this situation an ultrafilter $\F \in \U(X)$ can be seen as a functional mapping parts of $X$ into $2$ which satisfies the condition of being a finitely additive probability measure.  Thus an ultrafilter is a primitive sort of probability measure which is ``deterministic'' as  it assumes only the values $0$ or $1$. 
This basic construction can be generalized to the category of measurable spaces $\M$ so that a probability measure $P$ can be viewed as a measurable function $P \in \M( 2^X, \I)$ where $\I=[0,1]$ endowed with the Borel $\sigma$-algebra generated by the open intervals,  and $2^X = \M(X, 2)$  where $2$ has the discrete $\sigma$-algebra.   Because the characteristic functions on $X$ completely determine the set of measurable functions on $X$, by endowing the space $\I^X=\M(X,\I)$ with the  $\sigma$-algebra generated by all  the evaluation maps $ev_x: \I^X \rightarrow \I$, a probability measure $P$ on $X$ can be  viewed as an element of the double dualization monad $\M(\I^{\bullet}, \I)$ at $X$.   Thus, analogously   with the ultrafilter, we are led to construct a  submonad $\T$ of the  double dualization monad $\M(\I^{\bullet}, \I)$ so that a probability measure on $X$ corresponds to an element in $\T(X)$.   This (sub)monad $\T$ is naturally isomorphic to the Giry monad $\G$ where $\G(X)$ is also the space of all probability measures on $X$ \cite{Giry} and, for $f:X \rightarrow Y$ a measurable function, $\G(f): \G(X) \rightarrow \G(Y)$ is the push forward map, defined for all $P \in \G(X)$ by $P \mapsto P f^{-1}$.  In this introductory section we outline this procedure.

  To define this monad $\T$ componentwise we first recall that a functional $G: \I^X \rightarrow \I$ is \textbf{weakly averaging} when,  for all constant functions $\overline{u} \in \I^X$ with value $u=\overline{u}(x) \in \I$  for all $x \in X$, $G$ satisfies $G(\overline{u})=u$.\footnote{This terminology is taken from \emph{Sets for Mathematics} \cite{LR} who specifically address the double dualization  process and subfunctors thereof into objects with extra structure.} 
Also, as  $\I$ has a natural convex structure defined by $u +_{\alpha} v = \alpha u + (1-\alpha) v$ for all $\alpha \in \I$, the function space $\I^X$ has a convex structure defined on it pointwise.  Thus for $f,g \in \I^X$ the ``convex sum''   $f +_{\alpha} g$ is defined pointwise for each $x \in X$  by $f(x) +_{\alpha} g(x)$ making $f +_{\alpha} g \in \I^X$.   
The functional $G:\I^X \rightarrow \I$ is \textbf{affine}\footnote{The set of all affine morphisms $\I^X \rightarrow \I$ is also an object in the category of convex spaces, $\C$, where the morphisms have  been called  affine, affine linear, convex linear, as well as convex.} when $G(f +_{\alpha} g) = G(f) +_{\alpha} G(g)$ for all $f, g \in \I^X$ and all $\alpha \in \I$.  The functional $G:\I^X \rightarrow \I$ \textbf{preserves limits}  if for all $f \in \I^X$ the condition \mbox{$G(f) = \sup_{\psi \in \I^X}  \{G(\psi) \, | \, \psi \textrm{ simple}, \, \psi \le f \}$} is satisfied where $\psi \le f$ is pointwise inequality, $\psi(x) \le f(x)$ for all $x \in X$.   The property of preservation of limits plays the role of extending the affine property of  $G$, \mbox{$G(f +_{\alpha} g)= G(f) +_{\alpha} G(g)$},   
to a countably affine condition $G( \sum_{i=1}^{\infty} \alpha_i f_i) = \sum_{i=1}^{\infty} \alpha_i G(f_i)$ where $\lim_{i=1}^{\infty} \{ \sum_{j=1}^i \alpha_i\}=1$.  

Defining the monad $\T$ on objects, at  component $X$, by
\be \nonumber
\T(X) = \{ \I^X \stackrel{G}{\longrightarrow} \I \, | \, G \textit{ is weakly averaging, affine, and preserves limits} \}
\ee
along with the obvious corresponding map on arrows, 
we show this  monad is naturally isomorphic to the Giry monad $\G$.   For $\C$ the category of convex spaces and affine morphisms, we construct a functor $\iota:\C \rightarrow \M$ and show that the right Kan extension of $\iota$ along $\iota$ is the monad $\T$.  This functor $\iota$ is itself a subfunctor of a double dualization functor using the unit interval $\I$, which is an object in both of the categories, $\C$ and $\M$.

In proving the equivalence of the monads, $\T$ and $\G$,  we invoke the work of \cite{Kock} to prove the result that the integral operator 
\be \nonumber
 \begin{tikzpicture}[baseline=(current bounding box.center)]
 	\node	(IX) 	at	(0,0)              {$\I^X \otimes \G(X)$};
	\node         (I)  at    (4,-0)           {$\I$};
	\node         (f)   at    (0,-1)          {$(P,f)$};
	\node         (Pf)    at   (4,-1)        {$\int_X f \, dP$};
	
	\draw[->,above] (IX) to node {$\int_X \_ \, d\_$} (I);

	\draw[|->, right] (f) to node  [yshift=0pt]{$$} (Pf);
	
 \end{tikzpicture}
\ee
is in fact a measurable function, and the correspondence between the Giry monad $\G$ and the monad $\T$ amounts to saying there exist a correspondence sending a probability measure $P \in \G(X)$ to the affine weakly averaging functional $\int_X \_ \, dP$ which preserves limits, and conversely.  This result follows from the SMCC structure of $\M$ and the fact that the Giry monad is a strong monad.\footnote{Here we only prove the aspect of this property which is needed for our purposes.} 
  
In a previous version of this article the authors failed to include the condition of preserving limits which was brought to our attention  by Tom Avery, who has written an alternative viewpoint of the Giry monad as a codensity monad explicitly using integral operators which avoids using the SMCC arguments used here \cite{Avery}.  The reason we have chosen the more abstract and general approach is that we are concerned  with extending  the ideas of probability theory to categories other than $\M$  where the  subfunctor of the double dualization monad consisting of weakly averaging affine functionals which preserve limits is not a commutative monad  leading to noncommutative probability.   Further aspects to this, many of which should be evident from this work,  are left to future publications.

This paper is organized to sequentially show
\begin{enumerate}[noitemsep]
\item[(i)]  Both $\C$ and $\M$ are  symmetric monoidal closed categories (SMCC). 
\item[(ii)] For every measurable space $X$ the set of weakly averaging affine morphisms which preserve limits,  $\T(X)$,  is  isomorphic to $\G(X)$ as \emph{convex} spaces.
\item[(iii)]  This isomorphism of convex spaces extends to an  isomorphism of monads,  $\G \cong  \T$.
\item[(iv)]  There exist a  functor $\iota: \C \rightarrow \M$. 
\item[(v)]  The  monad $\T$ is a right Kan extension of $\iota$ along $\iota$. 
\end{enumerate}

\paragraph{Notation} Unless specifically defined otherwise, the symbols $X, Y,$ and $Z$ always denote measurable spaces while $A, B$, and $C$ always denote convex spaces.     The symbol $\I$ denoting the unit interval is of course both a measurable space and a convex space.  Although $\I^X = \M(X,\I)$ we still use the notation $\M(X,\I)$ when we want to think of a map $f:X \rightarrow \I$ as an actual function rather than an element in the function space $\I^X$. Formally, using the SMCC structure of $\M$, these two spaces $\I^X$ and $\M(X,\I)$ are only isomorphic and hence the distinction.  In the last section where a slice category is also used, in addition to the categories $\M$ and $\C$, it is convenient to use the notation ``$X \in_{ob} \mathcal{C}$'' to denote an object in the category $\mathcal{C}$ and ``$f \in_{ar} \mathcal{C}$'' to denote an arrow in the category $\mathcal{C}$.  For an object $X$ in any category the identity arrow on $X$ is denoted $id_X$.  The notation $\overline{u}$ is used to denote a constant function with value $u$ lying in the codomain of the function $\overline{u}$.

\section{The categories of interest}
The two main categories of interest are $\M$ and $\C$.  While most of the categorical properties of $\M$ are well known the fact that $\M$ is a SMCC is apparently not well know and hence we give an overview of this fact.\footnote{We are not aware of this fact in the literature though it would be surprising that it is not known as its construction is similar to that used in topology.}    We first provide a brief summary of $\C$ which is also a SMCC  \cite{Meng}, and provide a brief overview of that construction.   A more detailed description of the category of convex spaces can be found in  \cite{Fritz} who provides  definitions with numerous examples and highlights the difference between geometric and combinatorial convex spaces.  

\subsection{The category of convex spaces.}

A convex space $(A,+)$ consist of a set $A$ and a function 
\begin{equation}  \nonumber
\begin{array}{ccc}
A \times A \times I & \longrightarrow & A \\
(a_1,a_2,r) & \mapsto & a_1 +_r a_2
\end{array}
\end{equation}
satisfying the following axioms\footnote{In defining the convex structure of $A \times A$  an alternative representation, of the form
 $\sum_{i=1}^n r_i a_i$ with the relation $\sum_{i=1}^n r_i = 1$,  can be used rather than the free representation.  In this case the convex sum 
  $\sum_{i=1}^n r_i a_i=( \ldots ((a_1 +_{s_1} a_2) +_{s_2} a_3) + \ldots +_{s_{n-1}} a_n)$, where the elements $s_i \in \I$ for all $i=1,\ldots,n-1$ determine the coefficients $r_i$ and vice versa as these two representations are, assuming the $r_i \not = 0$, easily computed recursively.} 
\begin{equation}  \nonumber
\begin{array}{l}
(1)  \, a_1 +_0 a_2 = a_2 \\
(2) \, a +_r a = a  \\
(3) \, a_1 +_r a_2 = a_2 +_{1-r} a_1 \\
(4) \,  (a_1 +_p a_2) +_q a_3 = a_1 +_{pq} (a_2 +_r a_3) \quad \textrm{for }  r = \left\{ \begin{array}{ll} \frac{(1-p) q}{1-(pq)} & \textrm{if }pq \ne 1 \\ arbitrary & \textrm{if }p=q=1 \end{array} \right.
\end{array}
\end{equation}  
The convex structure of the convex space $\I$  is defined by $a_1 +_{r} a_2=r a_1 + (1-r) a_2$ for all \mbox{$a_1,a_2,r \in \I$}.

 An affine morphism of convex spaces $f:(A,+) \rightarrow (B,\oplus)$ satisfies
\be \nonumber
f( a_1 +_r a_2) = f(a_1) \oplus_r f(a_2).
\ee
These objects and morphisms determine the category of convex spaces $\C$.
If $A$ and $B$ are convex spaces we denote the set of all affine morphisms from $A$ to $B$ by $\C(A,B)$. 

The notation  $(a_1 +_{r_1} a_2)+_{r_2} a_3$ is   the \emph{free representation}  of ``convex sums'', wherein the values $r_1,r_2  \in \I$ have no further relationship.  The alternative and more familiar notation $\sum_{i=1}^3 s_i a_i$, subject to $\sum_{i=1}^3$, is the \emph{barycentric representation}.  The relationship between these two representations is elementary and we freely use both representations in this paper.

\subsection{The symmetric monoidal closed structure of $\C$.} \label{tensorConstruction}   

The unit of the SMCC structure on $\C$ is the object $1 = \{\star\}$ with the only possible convex structure.  The construction of the tensor product and function spaces in $\C$ is virtually identical to the construction employed in the category of modules over a ring, $\mathbf{R}$-$\mathbf{Mod}$.  Hence we limit ourself to reminding the reader of the basic construction.

The tensor product of two convex spaces $A \otimes B$  is obtained by taking the free convex structure on $A \times B$ and then taking the smallest congruence relation on this set such that
\be \nonumber
\sum_{i=1}^n \alpha_i (a_i, b) \equiv (\sum_{i=1}^n \alpha_i a_i, b)  \quad \textrm{and} \quad \sum_{i=1}^n \alpha_i (a, b_i) \equiv (a, \sum_{i=1}^n \alpha_i b_i). 
\ee
This tensor product $A \otimes B$ is universal in the sense that if $C$ is any convex space and $f:A \times B \rightarrow C$ is a bi-affine function (affine in each variable), then there exist a unique affine morphism $\hat{f}$ such that the diagram 

\be \nonumber
 \begin{tikzpicture}[baseline=(current bounding box.center)]
 	\node	(AxB)	at	(0,0)              {$A \times B$};
	\node	(AB)	at	(2,0)	               {$A \otimes B$};
	\node	(C)	at	(2,-1.8)               {$C$};

	\draw[->, above] (AxB) to node  {$$} (AB);
	\draw[->,left] (AxB) to node [xshift=-4pt,yshift=0pt] {$f$} (C);
	\draw[->,right] (AB) to node [xshift = 1pt] {$\hat{f}$} (C);

 \end{tikzpicture}
\ee
\noindent
commutes.

The convex structure on $B^A = \C(A,B)$ is defined pointwise.  If $f,g \in \C(A,B)$ then $(f +_{\alpha}g)(a) = f(a) +_{\alpha} g(a)$.  

Using these definitions the defining property of a closed monoidal category, $\_ \otimes B \vdash \_^B$ for all convex spaces $B$,  follows.  The symmetry follows from the construction of the tensor product $\otimes$.

\subsection{The symmetric monoidal closed structure of $\M$.}  

Throughout this section, $X$ and $Y$ denote measurable spaces.
The category $\M$  is a SMCC with the tensor product $X \otimes Y$  defined by the coinduced (final) $\sigma$-algebra such that all the  graph functions
\be  \nonumber
\begin{array}{ccccc}
\Gamma_f &:& X & \longrightarrow & X \times Y \\
&:& x & \mapsto & (x,f(x))
\end{array}
\ee 
for $f:X \rightarrow Y$ a measurable function, as well as the graph functions
\be  \nonumber
\begin{array}{ccccc}
\Gamma_g &:& Y & \longrightarrow & X \times Y \\
&:& y & \mapsto & (g(y),y)
\end{array}
\ee
for $g:Y \rightarrow X$ a measurable function, are measurable. 
 
Let $Y^X$ denote the set of all measurable functions from $X$ to $Y$ endowed with the $\sigma$-algebra induced by the set of all point evaluation maps\footnote{This is equivalent to saying $Y^X$ has the product $\sigma$-algebra induced by the coordinate projection maps onto $Y$.} 
\be \nonumber
\begin{array}{ccc}
Y^X &\stackrel{ev_x}{\longrightarrow}& Y \\
\ulcorner f \urcorner & \mapsto & f(x) 
\end{array}
\ee
\noindent
where the  notation  $\ulcorner f \urcorner$ is used to distinguish between the measurable function \mbox{$f:X \rightarrow Y$} and the point \mbox{$\ulcorner f \urcorner:1 \rightarrow Y^X$} of the function space $Y^X$.  
After establishing the SMCC property we drop the distinction as it is common practice to let the context define which arrow we are referring to.   

   Because the $\sigma$-algebra structure on tensor product spaces is defined  such that the  graph functions are all measurable, it follows in particular the constant graph functions \mbox{$\Gamma_{\ulcorner f \urcorner}:X \rightarrow X \otimes Y^X$} sending $x \mapsto (x,\ulcorner f \urcorner)$ are measurable.

Define the evaluation function
 \be \nonumber
 \begin{array}{ccc}
 X \otimes Y^X & \stackrel{ev_{X,Y}}{\longrightarrow}& Y \\
 (x,\ulcorner f \urcorner) & \mapsto & f(x)
 \end{array}
 \ee
and observe that for every $\ulcorner f \urcorner  \in Y^X$ the right hand diagram in the $\M$ diagrams

\be  \nonumber
 \begin{tikzpicture}[baseline=(current bounding box.center)]
 	\node	(X)	at	(0,-1.8)              {$X \cong X \otimes 1$};
	\node	(XY)	at	(0,0)	               {$ X \otimes Y^X$};
	\node	(Y)	at	(3,0)               {$Y$};
         \node         (1)    at      (-4,-1.8)             {$1$};
         \node         (YX) at      (-4,0)              {$Y^X$};

	\draw[->, left] (X) to node  {$\Gamma_{\ulcorner f \urcorner}\cong  Id_X \otimes \ulcorner f \urcorner$} (XY);
	\draw[->,below, right] (X) to node [xshift=3pt,yshift=-1pt] {$f$} (Y);
	\draw[->,above] (XY) to node {$ev_{X,Y}$} (Y);
	\draw[->,left] (1) to node {$\ulcorner f \urcorner$} (YX);

 \end{tikzpicture}
 \ee
\noindent 
is commutative as a set mapping, $f = ev_{X,Y} \circ \Gamma_{\ulcorner f \urcorner}$.   
Rotating the above diagram and also considering the constant graph functions $\Gamma_{x}$,  the right hand side of the  diagram 
\be \nonumber 
 \begin{tikzpicture}[baseline=(current bounding box.center)]
        \node          (X)    at      (-3,0)             {$X$};
 	\node	(YX)	at	(3,0)              {$Y^X$};
	\node	(XY)	at	(0,0)	               {$ X \otimes Y^X$};
	\node	(Y)	at	(0,-1.8)               {$Y$};

         \draw[->,above] (X) to node {$\Gamma_{\ulcorner f \urcorner}$} (XY);
         \draw[->,left] (X) to node [xshift=-7pt] {$f$} (Y);
	\draw[->, above] (YX) to node  {$\Gamma_{x}$} (XY);
	\draw[->,right] (YX) to node [xshift=5pt,yshift=0pt] {$ev_x$} (Y);
	\draw[->,right] (XY) to node {$ev_{X,Y}$} (Y);

 \end{tikzpicture}
 \ee
\noindent 
also commutes for every $x \in X$.
Since $f$ and $\Gamma_{\ulcorner f \urcorner}$ are measurable, as are $ev_{x}$ and $\Gamma_{x}$,  it follows by
the elementary result on final $\sigma$-algebras 
   
\begin{lemma} \label{coinduced} Let $Y$ have the  final  $\sigma$-algebra induced by the maps \mbox{$\{f_j : X_j \rightarrow Y \}_{j \in J}$}.   Then a function $g:Y \rightarrow Z$ is measurable if and only if the composition $g \circ f_j$ is measurable for each $j \in J$.
\end{lemma}
\noindent
that $ev_{X,Y}$ is measurable because the collection of all constant graph functions generate the $\sigma$-algebra of $X \otimes Y^X$.  
 
More generally, given any measurable function $f:X \otimes Z \rightarrow Y$ there exists a unique measurable map $\tilde{f} : Z  \rightarrow  Y^X$  defined by $\tilde{f}(z) = \ulcorner f(\cdot,z) \urcorner: 1 \rightarrow Y^X$ where $f(\cdot,z): X \rightarrow Y$ sends $x \mapsto f(x,z)$.
This map $\tilde{f}$ is measurable because the $\sigma$-algebra is generated by the \emph{point evaluation} maps $ev_x$ and the diagram

\be  \nonumber
 \begin{tikzpicture}[baseline=(current bounding box.center)]
        \node          (XZ)    at      (3,-1.8)             {$X \otimes Z$};
 	\node	(YX)	at	(0,0)              {$Y^X$};
	\node	(Y)	at	(3,0)               {$Y$};
	\node	(Z)	at	(0,-1.8)	               {$Z$};

         \draw[->,above] (YX) to node {$ev_x$} (Y);
         \draw[->,left] (Z) to node [xshift=-3pt] {$\tilde{f}$} (YX);
	\draw[->, above] (Z) to node [xshift=5pt] {$\Gamma_{x}$} (XZ);
	\draw[->,right] (XZ) to node [xshift=5pt,yshift=0pt] {$f$} (Y);
	\draw[->,right,dashed] (Z) to node {$$} (Y);

 \end{tikzpicture}
\ee
\noindent 
commutes for every $x\in X$ so by the dual of Lemma~\ref{coinduced} it follows that $\tilde{f}$ is measurable. 

Conversely given any measurable map $g  : Z \rightarrow Y^X$ it follows that the composite 
\mbox{$ev_{X,Y} \circ (Id_X \otimes g)$}
is a measurable map. This determines a bijective correspondence 
\be \nonumber
\M(X \otimes Z,Y) \cong \M(X,Y^Z).
\ee

Thus, for every measurable space $X$,  we have the adjunction $X \otimes \_ \dashv -^X$, and  the unit of the adjunction is the graph function $\Gamma_{\_}$ which is the unique map such that the diagram
\be  \nonumber
 \begin{tikzpicture}[baseline=(current bounding box.center)]

	\node	(YXXX)	at	(0,0)	               {$ X \otimes (X \otimes Y)^X$};
	\node	(YX)	at	(5,0)               {$X \otimes Y$};
	\node        (YX2)  at   (0,-1.8)          {$X \otimes Y$};
	
         \node         (Y)    at      (-4,-1.8)             {$Y$};
         \node         (YXX) at      (-4,0)              {$(X \otimes Y)^X$};

	\draw[->, left] (YX2) to node  {$id_X \otimes \Gamma_{\_}$} (YXXX);
	\draw[->,below, right] (YX) to node [xshift=4pt,yshift=-2pt] {$id_{X \otimes Y}$} (YX2);
	\draw[->,above] (YXXX) to node {$ev_{X,Y\otimes X}$} (YX);
	\draw[->,left] (Y) to node {$\Gamma_{\_}$} (YXX);

 \end{tikzpicture}
 \ee
\noindent 
commutes.  The graph function $\Gamma_{\_}$ is, like the unit of any adjoint pair,  defined by $\Gamma_{\_}(f) = \Gamma_f$, which is the constant graph function $\Gamma_f :  X \rightarrow (X \otimes Y)$.

\paragraph{Double dualization into the unit interval $\I$}   Recall  $\I=[0,1]$ with the Borel $\sigma$-algebra on the topology $\tau_{\I}$ generated by the open intervals, $\B_{\I} = \sigma( \tau_{\I})$.  As the function space $\I^X$ has the product \mbox{$\sigma$-algebra} it follows that each of the 
 point evaluation maps $\I^X \stackrel{ev_x}{\longrightarrow} \I$, for every $x \in X$,  is measurable.

We require the following elementary result for subsequent constructions.
\begin{lemma}  \label{measI}
Given any measurable space $X$ the double dual mapping\footnote{In this diagram and those to follow we abuse notation following the doctrine of  expressing the mapping into a function space not as the name of an element, like $\ulcorner ev_x \urcorner \in \I^{\I^X}$ for the given map $\eta_X(x)$,  but rather as the morphism corresponding to the named element.  The dashed arrow notation is employed to make it easier to read given the multiple arrows involved.}
\begin{equation}   \nonumber
 \begin{tikzpicture}[baseline=(current bounding box.center)]
         \node  (PA)  at (0,0)    {$X$};
         \node  (I)     at    (4,0)     {$\I^{I^X}$};

         \node  (PA2)  at (0,-.7)    {$x$};
         \node  (I2)     at    (4,-.7)     {$\I^X \stackrel{ev_x}{\longrightarrow} \I$};
         
	\draw[->,above] (PA) to node {$\eta_{X}$} (I);
	\draw[|->,densely dashed] (PA2) to node {} (I2);
	
\end{tikzpicture}
 \end{equation}
 \noindent
is a measurable function.
\end{lemma}
\proof
Since the functions $\{ev_f\}_{f \in \M(X,\I)}$ generate $\sa_{\I^{\I^X}}$ it suffices to show that \\
\mbox{$\eta^{-1}_{X}(ev_f^{-1}(U)) \in \sa_X$ for $U \in \sa_{\I}$}.  But this set is just $f^{-1}(U)$ which is measurable since $f$ is measurable.
\endproof

\section{The submonad of the double dualization monad}

Using the SMCC category structure we have 
 the double dualization monad $\M(\I^{\bullet},\I)$ on $\M$
specified by
\be  \nonumber
\begin{array}{lcccc}
\M(\I^{\bullet},\I) &:_{ob}&  X  & \mapsto &  \M(\I^X, \I) \\
&:_{ar} & X \stackrel{f}{\longrightarrow} Y & \mapsto &  \M(\I^X, \I)  \stackrel{\M(\I^f,\I)}{\longrightarrow} \M(\I^Y, \I)
\end{array}
\ee
where $\M(\I^f,\I)(G)= G\circ I^f$ is the pushforward of $G$ by $f$,
\begin{figure}[H]
\begin{equation}   \nonumber
 \begin{tikzpicture}[baseline=(current bounding box.center)]
         \node  (AY)  at (0,-2)    {$\I^Y$};
         \node  (AX)  at (0,.0)    {$\I^X$};
          \node  (A)  at  (2,0)      {$\I$};
          
	\draw[->,left] (AY) to node [xshift=-5pt] {$\I^{f}$} (AX);

	\draw[->, right,dashed]  (AY) to node [xshift=5pt]  {$G \circ \I^f : h  \mapsto G(h \circ f) \quad \forall h \in \I^Y$} (A);
	\draw[->, above] (AX) to node {$G$} (A);
	 \end{tikzpicture}
 \end{equation}
 \caption*{Diagram 1.   The pushforward of $G$ by $f$.}
  \label{pushforward}
 \end{figure}

The double dualization monad, similar to any double dualization monad on a SMCC, has the unit $\eta$ and multiplication $\mu$ defined componentwise by
\be \nonumber
 \begin{tikzpicture}[baseline=(current bounding box.center)]
 	\node	(X) 	at	(0,0)              {$X$};
	\node	(P)	at	(3,0)	               {$\M(\I^{X},\I)$};
	\node        (pt)   at     (0,-.6)       {$x$};
	\node        (IX)  at   ( 2.4,-.6)       {$\I^X$};
	\node        (I)    at    ( 3.8,-.6)      {$\I$};

	\draw[->, above] (X) to node  {$\eta_X$} (P);
	\draw[|->,densely dashed] (pt) to node {$$} (IX);
	\draw[->,below] (IX) to node {$ev_x$} (I);
	
	 \node	(PPX) 	at	(7,0)              {$\M(\I^{\M(\I^{X},\I)},\I)$};
	\node	(PX)	at	(12,0)	               {$\M(\I^{X},\I)$};
	\node        (Q)   at     (7,-.6)       {$Q$};
	\node        (IX2)  at   ( 10.4,-.6)       {$\I^X$};
	\node        (I2)    at    ( 13,-.6)      {$\I$};
	
	\node    (f)  at  (10.4,-1.7)  {$f$};
	\node    (Qf)  at  (13,-1.7)  {$Q(ev_f)$};

	\draw[->, above] (PPX) to node  {$\mu_X$} (PX);
	\draw[|->,densely dashed] (Q) to node {$$} (IX2);
	\draw[->,below] (IX2) to node {$\mu_X(Q)$} (I2);
	\draw[|->] (f) to node {} (Qf);

 \end{tikzpicture}
\ee
\noindent
Observe that the multiplication map at each component $X$, $\mu_X$,  can itself be viewed as a pushforward map   because for every \mbox{$Q \in \M(\I^{\M(\I^{X},\I)},\I)$} the diagram 
\begin{figure}[H]
\begin{equation}   \nonumber
 \begin{tikzpicture}[baseline=(current bounding box.center)]
         \node  (AY)  at (0,2)    {$\I^{\I^{\I^X}}$};
         \node  (AX)  at (0,.0)    {$\I^X$};
          \node  (A)  at  (2,2)      {$\I$};
          
	\draw[->,left] (AX) to node [xshift=-5pt] {$\eta_{\I^X}$} (AY);

	\draw[->, above]  (AY) to node [xshift=5pt]  {$Q$} (A);
	\draw[->, below,right] (AX) to node {$\mu_X(Q) = Q \circ \eta_{\I^X}$} (A);
	 \end{tikzpicture}
 \end{equation}
 \caption*{Diagram 1.5.   The multiplication map at component $X$ as a pushforward map.}
  \label{pushforward}
 \end{figure}
\noindent
commutes.

Note that the function space $\I^X$  has a convex structure associated with it, defined pointwise by $(f+_{\alpha}g)(x) = f(x) +_{\alpha} g(x)$ for all $f,g \in \I^X$.

Define the subfunctor $\T \hookrightarrow \M(\I^{\bullet},\I)$ componentwise by 
\be \nonumber
\T(X) = \{  \I^X \stackrel{G}{\longrightarrow} \I \, \, | \, G \textrm{ is a weakly averaging, affine, and preserves limits} \}.
\ee
This is a submonad of  $\M(\I^{\bullet},\I)$ because for any measurable space $X$ all the evaluation maps $ev_x$ are affine and weakly averaging by the pointwise convex structure on $\I^X$, and for every $x \in X$ the evaluation map $ev_x$ preserves limits.  
Similarly, if $Q \in \T(\T(X))$ then $\mu_X(Q) \in \T(X)$ because 
\begin{enumerate}[label=\roman*.]
\item $\mu_X(Q)$ is weakly averaging because for the constant function $\overline{c}:X \rightarrow \I$, using the fact each $G \in \T(X)$ is weakly averaging, it follows the function $ev_{\overline{c}}:\T(X) \rightarrow \I$ is equal to the constant function $\overline{c}: \T(X) \rightarrow \I$.  Therefore $\mu_X(Q)(\overline{c}) = Q(ev_{\overline{c}}) = Q(\overline{c}) = c$.
\item $ev_{f +_{\alpha} g} = ev_f +_{\alpha} ev_g$ which makes $\mu_X(Q)$  affine because $Q$ is affine.
\item   $\mu_X(Q)$  preserve limits because given a sequence of simple measurable functions $\{f_i\}_{i=1}^{\infty} \rightarrow f$  it follows $\{Q(ev_{f_i})\}_{i=1}^{\infty} \rightarrow Q(ev_f)$ because $Q$ preserves limits.
\end{enumerate}

Because  every  map $h \in \I^X$ can be written as the limit of a sequence of simple measurable functions it follows  the pushforward map of $G \in \T(X)$ along $f: X \rightarrow Y$, $\T(f)(G)$,  also preserve limits because 
\be  \nonumber
\begin{array}{lcl}
\T(f)(G)[g] &=& \sup_{\psi \in \I^Y} \{ G(\psi \circ f) \, | \, \psi \textrm{ simple}, \psi \le g\} \\
&=& \sup_{\phi \in \I^X} \{G(\phi) \, | \, \phi \textrm{ simple},   \phi \le (g \circ f)\} \\
&=& G(g \circ f)
\end{array}.
\ee

\section{Probability measures as weakly averaging affine functionals which preserve limits} 
Throughout this section, as well as subsequent sections, let $X$ denote a measurable space.  For any subset $S$ of $X$ we denote its complement by $S^c$.

\begin{lemma}\label{wellDefined}  Every simple measurable function $f:X \rightarrow \I$ can be written as a convex sum, $f = \sum_{i=1}^n a_i \chi_{S_i}$ with $\sum_{i=1}^n a_i =1$.
\end{lemma}
\begin{proof}  We  can assume the simple  measurable function  \mbox{$f = \sum_{i=1}^n a_i \chi_{S_i}$} is written with 
 pairwise disjoint measurable sets $\{S_i \}_{i=1}^n$  and has increasing coefficients, $a_1 \le a_2, \ldots \le  a_n$.  This sum can be rewritten as the ``telescoping'' function
\be  \nonumber
f = a_1 \chi_{\cup_{i=1}^n S_i} +(a_2 - a_1) \chi_{\cup_{i=2}^n S_i}  + \ldots  + (a_j - a_{j-1}) \chi_{\cup_{i=j}^n S_i} + \ldots  + (a_n - a_{n-1}) \chi_{S_n} + (1-a_n) \chi_{\emptyset}
\ee
which satisfies the condition that the sum of the coefficients is one, and each coefficient of this expression is easily seen to lie in the interval $\I$.
\end{proof}

As every measurable function $f: X \rightarrow \I$ can be written as the  limit of a sequence of simple measurable functions we obtain

\begin{cor} \label{convex} Every measurable function $f:X \rightarrow \I$ can be written as the limit of a sequence of simple measurable functions $\{f_i\}_{i=1}^{\infty}$ with each $f_i$ a convex sum. 
\end{cor}

   From this perspective, the motivation for the definition of  an affine  functional \mbox{$G: \I^X \rightarrow \I$} preserving limits is evident because, for  each simple function $f_i = \sum_{j=1}^{N_i} \alpha_{i,j} \chi_{A_{i,j}}$ with $\sum_{j=1}^{N_i} \alpha_{i,j}=1$, it follows  $G$ satisfies $G(f_i) = \sum_{j=1}^{N_i} \alpha_{i,j} G(\chi_{A_{i,j}})$.  
 
\begin{lemma} \label{basic}
For $G \in \T(X)$ and $\chi_S, \chi_T: X \rightarrow \I$  the characteristic functions associated with $S, T \in \sa_X$ it follows 
\begin{enumerate}
\item[(i)] $G(\chi_X) = 1$  and  $G(\chi_{\emptyset}) = 0$
\item[(ii)] $G(\chi_{S^c}) = 1 - G(\chi_{S})$
\item[(iii)] $G(\chi_{S \cap T}) + G(\chi_{S \cup T}) = G(\chi_{S}) + G(\chi_{T})$
\item[(iv)] If $S \subseteq T$  then $G(\chi_S) \le G(\chi_T)$
\item[(v)] If $\{S_i\}_{i=1}^{\infty}$ is a disjoint cover of $S$ by measurable sets then 
\be \nonumber
G(\chi_S) = \displaystyle{ \lim_{N \rightarrow \infty} \{ \sum_{i=1}^N G(\chi_{S_i}) \}}
\ee
\item[(vi)]  For any $\alpha \in \I$ and $f \in \I^X$, $G(\alpha f) = \alpha G(f)$.
\end{enumerate}
\end{lemma}
\begin{proof} (i) Since $\chi_X$ and $\chi_{\emptyset}$ are constant functions the result follows from the weakly averaging condition. 
(ii)  Consider the constant function
\be \nonumber
\overline{\frac{1}{2}} = \frac{1}{2} \chi_S + \frac{1}{2} \chi_{S^c}: X \rightarrow \I.  
\ee
Since  $G \in \T(X)$ it follows that
\be \nonumber
 \frac{1}{2} = G(\overline{\frac{1}{2}}) = G(\frac{1}{2} \chi_S + \frac{1}{2} \chi_{S^c})=\frac{1}{2} (G( \chi_S) + G(\chi_{S^c}))
\ee
which implies $G( \chi_S) + G(\chi_{S^c})=1$ and hence the  result.
(iii) This is a consequence of the observation that for all $S,T \in \sa_X$  the equation  
\be \nonumber
\frac{1}{2}  \chi_{S \cup T} + \frac{1}{2}  \chi_{S \cap T} =\frac{1}{2}  \chi_S + \frac{1}{2}  \chi_T
\ee
holds and both the left and right terms are measurable functions $X \rightarrow \I$.  
Applying $G$ to both sides of this expression gives the result.  
(iv) Apply the weakly averaging affine morphism $G$ to both sides of the  equation 
\be \nonumber
\frac{1}{2} \chi_T + \frac{1}{2} \chi_{\emptyset} = \frac{1}{2} \chi_S + \frac{1}{2} \chi_{T \cap S^c}
\ee
and use the condition $G(\chi_{T \cap S^c}) \ge 0$.
(v) By part (iii) $G(\chi_{S_1 \cup S_2}) = G(\chi_{S_1}) + G(\chi_{S_2})$ since the $S_i$ are disjoint.  Iterating this gives
the  monotone increasing sequence $\{ G(\chi_{ \cup_{i=1}^N S_i})  \}_{i=1}^{\infty}$ which is bounded above by $G(\chi_S)$.
Because  $G$ preserves limits, and each $\chi_{\cup_{i=1}^N S_i}$ is a simple function with $\{\chi_{\cup_{i=1}^N S_i} \}_{N=1}^{\infty} \nearrow \chi_S$, the result follows.
(vi) For $\alpha \in \I$ and  $f \in \I^X$ we observe that $\alpha f  = \alpha f + (1-\alpha) \chi_{\emptyset}$,  and by using the affine property of $G$ and  result (i), to obtain
\be  \nonumber
G(\alpha f) = G(\alpha f + (1-\alpha) \chi_{\emptyset}) = \alpha G(f) + (1-\alpha) G(\chi_{\emptyset}) = \alpha G(f).
\ee
\end{proof}

To prove the result that the two monads $\G$ and $\T$ are naturally isomorphic it is first necessary to prove a basic result regarding the Giry monad.

\begin{lemma}  There is a natural transformation $\tau$ of the two bifunctors

\be \nonumber
 \begin{tikzpicture}[baseline=(current bounding box.center)]
 	\node	(MM) 	at	(0,0)              {${\M}^{op} \times \M$};
	\node	(M) 	at	(6,0)	   {$\M$};

	\draw[->,above] ([yshift=2pt] MM.east) to node [yshift=0pt]{$\G(\_) \otimes \_$} ([yshift=2pt] M.west);
	\draw[->,below] ([yshift=-2pt] MM.east) to node {$\G(\_ \otimes \_)$} ([yshift=-2pt] M.west);
	
 \end{tikzpicture}
\ee
\noindent
defined component wise  for all $y \in Y$ and $P \in \G(X)$ by \mbox{$\tau_{P,y} = P \Gamma_{y}^{-1}$}, where
 \mbox{$\Gamma_y: X \rightarrow X \otimes Y$}  is the constant graph function.
\end{lemma}
\begin{proof}
Let $f: X\rightarrow X'$ and $g:Y \rightarrow Y'$.    The naturally condition follows easily using the diagram
\be \nonumber
 \begin{tikzpicture}[baseline=(current bounding box.center)]
 	\node	(GXY) 	at	(0,0)              {$\G(X) \otimes Y$};
	\node	(GXY2) 	at	(4,0)	   {$\G(X \otimes Y)$};
	\node         (GXYp)     at   (0,-2)     {$\G(X) \otimes Y'$};
	\node      (GXYp2)    at   (4,-2)   {$\G(X \otimes Y')$};
	\node       (GXpYp)   at   (0,-4)   {$\G(X') \otimes Y'$};
	\node       (GXpYp2)  at   (4,-4)  {$\G(X' \otimes Y')$};

	\draw[->, above] (GXY) to node  {$\tau_{X,Y}$} (GXY2);
	\draw[->, above] (GXYp) to node  [yshift=0pt]{$\tau_{X,Y'}$} (GXYp2);
	\draw[->,above] (GXpYp) to node {$\tau_{X',Y'}$} (GXpYp2);
	\draw[->,left] (GXY) to node {$id_{\G(X)} \otimes g$} (GXYp);
	\draw[->,right] (GXY2) to node {$\G(id_X \otimes g)$} (GXYp2);
	\draw[->,left] (GXYp) to node {$\G(f) \otimes id_Y$} (GXpYp);
	\draw[->,right] (GXYp2) to node {$\G(f \otimes id_{Y'})$} (GXpYp2);

 \end{tikzpicture}
\ee
and the relations $(id_X \otimes g) \circ \Gamma_y = \Gamma_{g(y)}$ and $(f \otimes id_{Y'}) \circ \Gamma_{y'} = \Gamma_{y'} \circ f$.

To prove that $\tau_{X,Y}$ is measurable consider the following two diagrams
\be \nonumber
 \begin{tikzpicture}[baseline=(current bounding box.center)]
	\node	(XGY) 	at	(3,0)	   {$\G(X) \otimes Y$};
	\node        (GXY)    at   (7,0)           {$\G(X \otimes Y)$};
	\node        (I)           at   (10,0)           {$\I$};

	\draw[->,above] (XGY) to node {$\tau_{X,Y}$} (GXY);
	\draw[->,above] (GXY) to node {$ev_{\zeta}$} (I);
	
	\node   (Y)   at   (12,0)    {$X$};
	\node   (XY)  at   (14,0)   {$X \otimes Y$};
	\draw[->,above] (Y) to node {$\Gamma_y$} (XY);
	
 \end{tikzpicture}
\ee
and let $W \in \sa_{\I}$, and $\zeta \in \sa_{X \otimes Y}$.  Taking the preimage of $W$ under the evaluation map $ev_{\zeta}$ gives a measurable set in $\G(X \otimes Y)$.  Recall that sets of this form generate the smallest $\sigma$-algebra on $\G(X \otimes Y)$ such that the evaluation maps are measurable.  The preimage of this set under $\tau_{X,Y}$ then yields the set
\be \nonumber
\tau_{X,Y}^{-1}(ev_{\zeta}^{-1}(W)) = \{ (P, y) \in \G(X) \otimes Y \, | \, P \left(\Gamma_y^{-1}(\zeta)\right) \in W \}. 
\ee
Note that both the  projections maps $\pi_Y : \G(X) \otimes Y \rightarrow Y$ and $\pi_{\G(X)}: \G(X) \otimes Y \rightarrow \G(X)$ are measurable under the given $\sigma$-algebra on $\G(X) \otimes Y$ because $\sa_{\G(X) \otimes Y}$ contains the product $\sigma$-algebra.  If $y \not \in \pi_Y(\zeta)$ then $\Gamma_{y}^{-1}(\zeta) = \emptyset$ and $P(\emptyset)=0$.  Hence if $0 \not \in W$ then $\tau(P,y) \not \in W$ for any $P \in \G(X)$ so the inverse image is the empty set which is measurable.  Consideration of the case $y \in \pi_Y(\zeta)$ gives
\be \nonumber
\tau_{X,Y}^{-1}(ev_{\zeta}^{-1}(W)) = \displaystyle{ \cup_{y \in \pi_Y(\zeta)}}  ev^{-1}_{\Gamma_y^{-1}(\zeta)}(W) \times  \{ y \}
\ee
where the map
\be \nonumber
 \begin{tikzpicture}[baseline=(current bounding box.center)]
	\node	(GY) 	at	(0,0)	   {$ \G(X)$};
	\node        (I)    at   (4,0)           {$\I$};

	\draw[->,above] (GY) to node {$ev_{\Gamma_{y}^{-1}(\zeta)}$} (I);
	
 \end{tikzpicture}
\ee
yields the measurable sets $ev_{\Gamma_{y}^{-1}(\zeta)}^{-1}(W)$ in $\G(X)$ for all $y \in Y$.  To prove that the set \mbox{$\cup_{y \in \pi_Y(\zeta)} ev^{-1}_{\Gamma_y^{-1}(\zeta)}(W)  \times \{ y \}$} is measurable in $\G(X) \otimes Y$ it suffices to show that the preimage of this set under the constant graph functions
\be \nonumber
 \begin{tikzpicture}[baseline=(current bounding box.center)]
	\node	(GY) 	at	(0,0)	   {$ \G(X)$};
	\node        (XGY)    at   (3,0)           {$\G(X) \otimes Y$};
	\draw[->,above] (GY) to node {$\overline{\Gamma}_{z}$} (XGY);
	
		\node	(X) 	at	(6,0)	   {$Y$};
	\node        (XGY2)    at   (9,0)           {$\G(X) \otimes Y$};
	\draw[->,above] (X) to node {$\Gamma_Q$} (XGY2);

 \end{tikzpicture}
\ee
is measurable, for all $z \in Y$ and all $Q \in \G(X)$.\footnote{We use an overbear notation on the constant graph function $\G(X) \rightarrow \G(X) \otimes Y$ to distinguish it from the constant graph function $X \rightarrow X \otimes Y$ defined previously in the proof.} 

Fix $z \in Y$.   If $z \in \pi_Y(\zeta)$ then
\be \nonumber
\overline{\Gamma}_{z}^{-1}(\cup_{y \in \pi_Y(\zeta)}  ev^{-1}_{\Gamma_y^{-1}(\zeta)}(W) \times \{y\})  = ev^{-1}_{\Gamma_z^{-1}(\zeta)}(W)
\ee
which is measurable by the construction of the $\sigma$-algebra on $\G(X)$, and  if $z \not \in \pi_X^{-1}(\zeta)$ one obtains the empty set which is measurable.   

Similarly, fix $Q \in \G(X)$.  Then
\be \nonumber
\begin{array}{lcl}
\Gamma_Q^{-1}( \cup_{y \in \pi_Y(\zeta)} ev^{-1}_{\Gamma_y^{-1}(\zeta)}(W)) \times \{ y \}) &=& \{y \in \pi_Y(\zeta) \, | \, Q \in ev_{\Gamma_y^{-1}(\zeta)}^{-1}(W) \} \\
&=& \Gamma_Q^{-1}( \cup_{y \in \pi_Y(\zeta)}  \pi_{\G(X)}^{-1}(ev_{\Gamma_y^{-1}(\zeta)}^{-1}(W)) \times Y) \cap \zeta) \\
&=& \Gamma_Q^{-1}(  \cup_{y \in \pi_Y(\zeta)} \pi_{\G(X)}^{-1}(ev_{\Gamma_y^{-1}(\zeta)}^{-1}(W)) \times Y) \,  \cap \, \Gamma_Q^{-1}(\zeta) 
\end{array}
\ee
which is measurable because the first term is either $Y$ or $\emptyset$, while the second term is measurable because $\Gamma_Q$ is measurable and $\zeta$ is a measurable set.
The  above equation is motivated by consideration of the diagram
\be \nonumber
 \begin{tikzpicture}[baseline=(current bounding box.center)]
	\node	(I) 	at	(0,0)	   {$ \I$};
	\node        (GY)    at   (3,0)           {$\G(X)$};
	\node        (XGY)  at   (6,0)     {$\G(X) \otimes Y$};
	\node        (X)       at    (9,0)     {$Y$};
	\node         (X2)    at     (6,-1.75)   {$Y$};
	
	\draw[->,above] (GY) to node {$ev_{\Gamma_y^{-1}(\zeta)}$} (I);
	\draw[->,above] (XGY) to node {$\pi_{\G(X)}$} (GY);
	\draw[->,above] (XGY) to node {$\pi_Y$} (X);
	\draw[->,right] (X2) to node {$\Gamma_Q$} (XGY);

 \end{tikzpicture}
\ee
\end{proof}

\begin{lemma} \label{two} For $2=\{0,1\}$ with the discrete $\sigma$-algebra it follows  $\G(2) = \I$.
\end{lemma}
\begin{proof} The set of probability measure on $2$ is  $\G(2) = \{\delta_{\{0\}} +_{\alpha} \delta_{\{1\}} \, | \,\alpha \in \I\}$. Measurability of the obvious map $\G(2) \rightarrow \I$ is immediate since the $\sigma$-algebra on $\G(2)$ is defined such that the evaluation maps 
$ev_{\{0\}} : \G(2) \rightarrow \I$ and $ev_{\{1\}} : \G(2) \rightarrow \I$ are both measurable, so the $\sigma$-algebra on $\G(2)$ is generated by the sets $\{ \delta_{\{0\}} +_{\alpha} \delta_{\{1\}} \, | \,\alpha \in U\}_{U \in \sa_{\I}}$. 
\end{proof}

\begin{lemma} \label{structure} The function $st_{X,Y}: Y^X \rightarrow \G(Y)^{\G(X)}$ defined by $f  \mapsto \G(f)$ is a measurable function.
\end{lemma}
\begin{proof}Using the SMCC structure of $\M$  the map $st_{X,Y}$ can be written as a composite of three other measurable functions
\be \nonumber
 \begin{tikzpicture}[baseline=(current bounding box.center)]
 	\node	(YX) 	at	(0,0)              {$Y^X$};
	\node	(GYGX) 	at	(6,0)	   {$\G(Y)^{\G(X)}$};
	\node         (T1)     at   (0,-2)     {$(\G(X) \otimes Y^X)^{\G(X)}$};
	\node      (T2)    at   (6,-2)   {$\G(X \otimes Y^X)^{\G(X)}$};

	\draw[->, above] (YX) to node  {$st_{X,Y}$} (GYGX);
	\draw[->, left] (YX) to node  [yshift=0pt]{${\Gamma_{\_}}$} (T1);
	\draw[->,below] (T1) to node {$\tau_{X,Y^X}^{id_{\G(X)}}$} (T2);
	\draw[->,right] (T2) to node {$\G(ev_{X,Y})^{id_{\G(X)}}$} (GYGX);
		
 \end{tikzpicture}
\ee
\noindent
which follows from \cite{Kock}.  For every $f \in \I^X$,   $st_{X,Y}(f): \G(X) \rightarrow \G(Y)$ is the map
\be \nonumber
\begin{array}{ccccccc}
\G(X) & \stackrel{\Gamma_f}{\longrightarrow} & \G(X) \otimes Y^X & \stackrel{\tau_{\G(X),Y^X}}{\longrightarrow} & \G(X \otimes Y^X) & \stackrel{\G(ev_{X,Y})} {\longrightarrow} & \G(Y) \\
P & \mapsto & (P,f) & \mapsto &  P {\Gamma_f}^{-1} & \mapsto & P {\Gamma_f}^{-1} ev_{X,Y}^{-1} 
\end{array}.
\ee
\end{proof}

Taking $Y = \I$ in  Lemma~\ref{structure}, and using Lemma~\ref{two},  it follows that the  composite of the two  measurable maps
\be \nonumber
 \begin{tikzpicture}[baseline=(current bounding box.center)]
 	\node	(IX) 	at	(-1,0)              {$\I^X$};
	\node	(GIGX) 	at	(2,0)	   {$\G(\I)^{\G(X)}$};
	\node         (IGX)     at   (5,0)     {$\I^{\G(X)}$};
	\node      (c)    at   (8.5,0)   {$\hat{\int}_X \_ \, d\_ \stackrel{def}{=} \mu_2' \circ st_{X,\I}$};

	\draw[->, above] (IX) to node  {$st_{X,\I}$} (GIGX);
	\draw[->, above] (GIGX) to node  [yshift=0pt]{$\mu'_2$} (IGX);
	
 \end{tikzpicture}
\ee
\noindent
is  the measurable map satisfying 
\be \nonumber
\begin{array}{lcl}
\left( \hat{\int}_X \_ \, d\_ \right)(f) &=& (\mu_2' \circ st_{X,I})(f) \\
&=& \mu_2(P\Gamma_f^{-1} ev_{X,I}^{-1})(\{1\}) \\
&=& \displaystyle{\int_{\underbrace{\delta_0 +_u \delta_1 \in \G(2)}_{\cong  u \in \I}}} \underbrace{ev_{\{1\}}(\delta_0 +_{u} \delta_1)}_{\cong id_{\I}(u)} \, d(P\Gamma_f^{-1}ev_{X,I}^{-1})) \\
&=& \int_{X \otimes Y^X} ev_{X,I}(x,g) \, dP \Gamma_f^{-1} \\
&=& \int_X ev_{X,\I}(\Gamma_f(x)) \, dP \\
&=& \int_X f \, dP
\end{array}
\ee

\begin{thm}   The map
\be \nonumber
 \begin{tikzpicture}[baseline=(current bounding box.center)]
 	\node	(IX) 	at	(0,0)              {$\G(X) \otimes \I^X$};
	\node	(IGX) 	at	(4,0)	   {$\I$};
         \node         (fP)    at   (0,-1)    {$(P,f)$};
         \node         (int)   at    (4,-1)    {$\int_X f \, dP$};
         
	\draw[->,above] (IX) to node {$\int_X \_ \, d\_$} (IGX);

	\draw[|->] (fP) to node  {} (int);	
 \end{tikzpicture}
\ee
is measurable.  In particular, 
for $P \in \G(X)$ the integral operator \mbox{$\int_X \_ \, dP : \I^X \rightarrow \I$} is a measurable function.  
\end{thm}
\begin{proof}  This follows immediately from the SMCC structure of $\M$ and the measurability of $\hat{\int_X} \_ \, d\_ $.

For a fixed $P \in \G(X)$, using the measurability of the map $\hat{\int_X} \_ \, d\_ $ and the map $ev_P$, which is measurable by construction of the $\sigma$-algebra on $\I^{G(X)}$, it follows  the diagram
\be \nonumber
 \begin{tikzpicture}[baseline=(current bounding box.center)]
 	\node	(IX) 	at	(0,0)              {$\I^X$};
	\node	(IGX) 	at	(3,0)	   {$\I^{\G(X)}$};
	\node         (I)  at    (1.5,-2.)           {$\I$};

	\draw[->,above] (IX) to node {$\hat{\int_X} \_ \, d\_$} (IGX);

	\draw[->, right] (IGX) to node  [yshift=0pt]{$ev_P$} (I);
	\draw[->, left] (IX) to node  {$\int_X \_ \, dP$} (I);
	
 \end{tikzpicture}
\ee
\noindent
commutes.   
\end{proof}

We observe that the map $\hat{\int_X} \_ \, d\_$ is a section of the map $\I^{\eta_X'} : \I^{\G(X)} \rightarrow \I^X$ obtained from the unit of the Giry monad $\eta_X'$.

\begin{lemma} \label{basicThm} There exists an isomorphism of convex spaces 
 \be  \nonumber
 \begin{array}{ccc}
 \T(X) & \stackrel{\phi}{ \longrightarrow}  & \G(X) \\
G & \mapsto & \nu_G
\end{array}
\ee
where $\nu_G(S) = G(\chi_S)$ for all $S \in \sa_X$.
\end{lemma}
\begin{proof}
  The verification that $\nu_G$ defines a probability measure follows directly from the definition of $\nu_G$ in terms of $G$ and the characteristic functions by applying Lemma~\ref{basic} and the properties of $G$.  The only nontrivial aspect is the fact that, for any disjoint covering $\{S_i\}_{i=1}^{\infty}$ of a measurable set $S$ by measurable subsets, the property 
 \be \nonumber
 \begin{array}{lcll}
  \nu_G(S)&=&G(\chi_{S}) & \\
  &=&  \lim_{N \rightarrow \infty} \{ \sum_{i=1}^N G(\chi_{S_i}) \} &  \textrm{by Lemma~\ref{basic}(v) } \\
 &=& \lim_{N \rightarrow \infty} \{ \sum_{i=1}^N \nu_G(S_i) \} &
 \end{array}
 \ee
holds.

The inverse of $\phi$ is the map
 \be  \nonumber
 \begin{array}{ccc}
 \G(X) & \stackrel{\gamma}{ \longrightarrow}  & \T(X) \\
P & \mapsto & \int_X \_ \, dP 
\end{array}
\ee
which is measurable by the previous Theorem, and the function  $\int_X \_ \, dP$ has the three required properties
 (1) weakly averaging, (2) affine, and (3) preserves limits.   The property of preservation of limits follows from the monotone convergence theorem and Corollary~\ref{convex}.    

These two maps, $\phi$ and $\gamma$,  are inverses because for $G \in \T(X)$ and $f \in \I^X$ it follows
\be \nonumber
\begin{array}{lcll}
G(f) &=& \displaystyle{\sup_{ \psi \textrm{ simple}} }\{ G(\psi) \, | \, \psi \le f, \,  \psi \in \I^X \} & \textrm{by definition of } G \\ 
&=& \displaystyle{\sup_{ \psi \textrm{ simple} }} \{ G(\underbrace{\sum_{i=1}^n \alpha_i \chi_{A_i}}_{=\psi}) \, | \, \psi \le f, \,  \psi \in \I^X, \, \sum_{i=1}^n \alpha_i = 1\} & \text{ by Lemma~\ref{wellDefined}}\\
&=& \displaystyle{\sup_{ \psi \textrm{ simple} }} \{ \sum_{i=1}^n \alpha_i  G(\chi_{A_i}) \, | \, \psi \le f, \,  \psi \in \I^X,  \, \sum_{i=1}^n \alpha_i = 1\} & \textrm{affine property of }G \\
&=& \displaystyle{\sup_{ \psi \textrm{ simple} }} \{ \sum_{i=1}^n \alpha_i  \nu_G(A_i) \, | \, \psi \le f, \,  \psi \in \I^X,  \, \sum_{i=1}^n \alpha_i = 1\} & \textrm{def. of the map }\phi \\
&=& \displaystyle{\sup_{ \psi \textrm{ simple} }} \{ \int_X  \psi \, d\nu_G \, | \, \psi \le f, \,  \psi \in \I^X\} & \textrm{def. of the integral}\\
&=&  \int_X f \, d\nu_G  & \textrm{def. of the integral of }f\\
&=& \left( (\gamma \circ \phi)(G)\right) (f)
\end{array}
\ee
and for $P \in \G(X)$ and $A \in \sa_X$ it follows
\be \nonumber
\begin{array}{lcl}
P(A) &=& \int_X \chi_A \, dP \\
&=& \gamma(P)(\chi_A) \\
&=& \left( (\phi \circ \gamma)(P)\right)(A)
\end{array}
\ee
\end{proof}

\begin{thm}  \label{monadIso}
The isomorphism of convex spaces in Lemma~\ref{basicThm} extends to a natural isomorphism of monads \mbox{$\phi: \T \rightarrow \G$}. 
\end{thm}
\begin{proof}   
First we show that $\phi_X:\T(X) \rightarrow \G(X)$ is an isomorphism of measurable spaces which requires showing $\T(X)$ with its subspace \mbox{$\sigma$-algebra} is isomorphic to the \mbox{$\sigma$-algebra} on $\G(X)$.   Recall that the Giry monad is endowed with the smallest \mbox{$\sigma$-algebra} such that each of the evaluation maps $ev_S: \G(X) \rightarrow \I$ sending a probability measure $P \mapsto P(S)$ is  measurable, for every measurable set $S$ in $X$. 
On the other hand the function space $\I^{\I^X}$ has the smallest $\sigma$-algebra such that each of the evaluation maps
\be \nonumber
 \begin{tikzpicture}[baseline=(current bounding box.center)]
 	\node	(IIX) 	at	(0,0)              {$\I^{\I^X}$};
	\node	(I)	at	(3,0)	               {$\I$};

	\draw[->, above] (IIX) to node  {$ev_f$} (I);
 \end{tikzpicture}
\ee
\noindent
is measurable for every measurable function $f:X \rightarrow \I$, so for $U \in \B_{\I}$ it follows that the set 
\be \label{algebra}
ev_f^{-1}(U) = \{ \I^X \stackrel{G}{\rightarrow} \I  \, | \, G(f) \in U \}
\ee
is measurable in $\I^{\I^X}$, and sets of this form, as $f$ varies over $\I^X$ and $U$ varies over $\B_I$, form a generating set for the $\sigma$-algebra on $\I^{\I^X}$.  Being more economical it suffices to take the generating set on the characteristic functions $f=\chi_S$ for all $S \in \sa_X$.  Restriction of the $\sigma$-algebra generated by these elements $\{ ev_{\chi_S}^{-1}(U) \}_{S \in \sa_X, U \in \sa_I}$ to the subset $\T(X)$ gives the $\sigma$-algebra on 
 $\T(X)$.  Under the mapping $\phi$ the generating set elements in (\ref{algebra}) get mapped to the subsets of $\G(X)$ corresponding to the 
 preimage of the diagonal map in the diagram
 \be \nonumber
 \begin{tikzpicture}[baseline=(current bounding box.center)]
 	\node	(IIX) 	at	(0,0)              {$\T(X)$};
	\node	(I)	at	(3,0)	               {$\I$};
	\node         (PX)  at    (0,-1.7)           {$\G(X)$};

	\draw[->, above] (IIX) to node  {$ev_{\chi_S}$} (I);
	\draw[->, right] (PX) to node  [yshift=-5pt]{$\int_X \chi_S \, d\_ = ev_S$} (I);
	\draw[->, left] (IIX) to node  {$\phi_X$} (PX);
	
         \node	(IIX2) 	at	(6,0)              {$G$};
	\node	(I2)	at	(9,0)	               {$G(\chi_S) = \nu_G(S)$};
	\node         (PX2)  at    (6,-1.7)           {$\nu_G$};

	\draw[|->, above] (IIX2) to node  {$$} (I2);
	\draw[|->, right] (PX2) to node  [yshift=-3pt]{$$} (I2);
	\draw[|->, left] (IIX2) to node  {$$} (PX2);
 \end{tikzpicture}
\ee
\noindent
which  are the generating elements for the $\sigma$-algebra of $\G(X)$. 
The converse then follows similarly mapping the generating elements of $\G(X)$ to the generating elements of $\T(X)$.

To show naturally let  $f:X \rightarrow Y$ be a measurable function and consider the right hand diagram in
\begin{figure}[H] 
\begin{equation}   \nonumber
 \begin{tikzpicture}[baseline=(current bounding box.center)]
         \node  (X)   at  (0,0)     {$\T(X)$};
         \node  (Y)  at (0,-1.8)    {$\T(Y)$};
         \node  (IX)  at (3,0)   {$\G(X)$};
         \node  (IY)  at (3,-1.8)   {$\G(Y)$};
         \node  (X2)    at  (-2,0)    {$X$};
         \node  (Y2)   at  (-2,-1.8)   {$Y$};
                   
	\draw[->,left] (X2) to node {$f$} (Y2);
	\draw[->,left] (X) to node {$\T(f)$} (Y);
	\draw[->,right] (IX) to node [xshift=-2pt,yshift=0pt] {$\G(f)$} (IY);
	\draw[->,above] (X) to node {$\phi_X$} (IX);
	\draw[->,above] (Y) to node {$\phi_Y$} (IY);
	 \end{tikzpicture}
 \end{equation} 
 \caption*{Diagram 2.  The naturality of $\phi$.}
 \label{fig:naturality}
 \end{figure}
 \noindent
The map 
 $\T(f)$ is just the pushforward map
shown in \mbox{Diagram 1}, restricted to the domain $\T(X)$ with the codomain $\T(Y)$.  Consequently, for $G \in \T(X)$ and $S \in \sa_Y$, 
it follows the computation of the south-east path of  Diagram 2 
 yields 
 \be \label{arrowFunctor}
\begin{array}{lcl}
\left( \phi_Y(\T(f)(G)) \right)[S] &=& \nu_{G \circ \I^f}(S) \\
&=& (G \circ \I^f)(\chi_S)  \\
&=& G(\chi_S \circ f) \\
&=& G(\chi_{f^{-1}(S)}) \\
&=& \nu_G( f^{-1}(S)) 
\end{array}.
\ee
whereas the computation of the east-south path  gives
\be
\begin{array}{ccll}
\G(f)( \phi_X(G))[S] &=&  \G(f) (\nu_G) [S]&\\
&=& \nu_G  (f^{-1}(S))& 
\end{array}
\ee
and hence the map $\phi$ is a natural transformation.
This natural transformation 
has the inverse natural transformation specified in Theorem~\ref{basicThm}.

The natural isomorphism $\phi: \T \rightarrow \G$ is a morphism of monads as it makes the two requisite diagrams 

\be \nonumber
 \begin{tikzpicture}[baseline=(current bounding box.center)]
 	\node	(IdX) 	at	(1,0)              {$Id$};
	\node	(P)	at	(4,0)	               {$\T$};
	\node        (G)   at     (4,-1.8)       {$\G$};

	\draw[->, above] (IdX) to node  {$\eta$} (P);
	\draw[->,below,left] (IdX) to node [xshift=-3pt,yshift=-3pt] {$\eta'$} (G);
	\draw[->,right] (P) to node {$\phi$} (G);
	
	\node	(PP)	at	(7.5,0)	        {$\T \circ \T$};
	\node      (P2)   at      (11,0)     {$\T$};
	\node      (GG)  at      (7.5,-1.8)        {$ \G \circ \G$};
	\node      (G2)  at      (11,-1.8)     {$\G$};
	
	\draw[->,above] (PP) to node {$\mu$} (P2);
	\draw[->,left] (PP) to node {$\phi \cdot \phi$} (GG);
	\draw[->, below] (GG) to node {$\mu'$} (G2);
	\draw[->,right] (P2) to node {$\phi$} (G2);

 \end{tikzpicture}
\ee
\noindent
commute, where  the natural transformation $\phi \cdot \phi$ is defined at component $X$, for any for $Q \in \T(\T(X))$,  by $(\phi \cdot \phi)_X(Q) = \G(\phi_X)(\phi_{\T(X)}(Q)) = \phi_{\G(X)}(\T(\phi_X)Q)$.\footnote{The term $\phi \cdot \phi$ is the ``horizontal   composite'' of two natural transformations, and the given equality in the definition of the horizontal composite is an easy verification \cite[Page 42]{MacLane}.}  Recalling that the unit of the Giry monad is defined by $\eta'_X(x) = \delta_x$ while the counit is specified by $\mu'_X(Q)(S) = \int_{q \in \G(X)} q(S) \, dQ$ for all $S \in \sa_X$, the commutativity of the left diagram follows from
\be \nonumber
\begin{array}{lcl}
\left( \phi_X(\eta_X(x)) \right)(S) &=& \left(\phi_X(ev_x)\right)(S) \\
&=& \nu_{ev_x}(S) \\
&=& ev_x(\chi_S) \\
&=& \chi_S(x) \\
&=& \delta_x(S) \\
&=& \eta'_X(x)(S)
\end{array}
\ee
while the commutativity of the right diagram is established by computing that the east-south path yields
\be \nonumber
\begin{array}{lcl}
\left( \phi_X( \mu_X(Q)) \right)(S) &=&  \nu_{\mu_X(Q)}(S) \\
&=&  \mu_X(Q)(\chi_S)\\
&=& Q(ev_{\chi_S})
\end{array}
\ee
while computing the south-east path yields the same value because
\be \nonumber
\begin{array}{lcl}
((\mu_X^{'} \circ (\phi \cdot \phi)_{X})Q)(S) &=& \mu_X^{'} \left(\G(\phi_X)(\underbrace{\phi_{\T(X)}(Q)}_{=\nu_Q})\right)(S) \\
&=& \left( \mu_X^{'}(\nu_{Q } \circ\phi_X^{-1})\right)(S) \\
&=& \int_{q \in \G(X)} ev_S(q) \, d(\nu_{Q} \circ \phi_X^{-1}) \\
&=& \int_{p \in \T(X)}(\underbrace{ev_S \circ \phi_X}_{=ev_{\chi_S}})(p) \, d\nu_Q \\
&=& Q(ev_{\chi_S})
\end{array}.
\ee

\end{proof}

\section{Constructing the functor $\iota:\C \rightarrow \M$}    \label{sec:iotaMap}
 
For $A$ a convex space endow the set of functions $\Set(A,\I)$  with the initial $\sigma$-algebra generated by the set of evaluation maps  to obtain the measurable space $\I^A \stackrel{def}{=}(\Set(A,\I), \sa_{init})$
where $\sa_{init}$ denotes the initial $\sigma$-algebra generated by the set of evaluation maps, 
\be \nonumber
\{ \Set(A,\I) \stackrel{ev_a}{\longrightarrow} \I\, | a \in A\}.  
\ee
 However we use the notation $\I^A$ for uniformity as the emphasis in this construction is on the $\sigma$-algebra construction being identical to that used in the construction of function spaces in $\M$, and the fact that $\M(A,\I)$ is (generally) without meaning as $A$ has no $\sigma$-algebra associated with it. The object $\I^{\I}$ will always refer to $\M(\I,\I)$ with the $\sigma$-algebra making the evaluation maps measurable so no confusion should arise.\footnote{Recall, convex spaces are  denoted using $A,B$, and $C$ whereas measurable spaces are denoted using $X,Y$, and $Z$.}

 As sets, $\C(A,\I) \subseteq \Set(A,\I)$, so we can endow $\C(A,\I)$ with the  subspace $\sigma$-algebra of $\I^A$, i.e., endow $\C(A,\I)$ with the initial $\sigma$-algebra generated by the inclusion map $\C(A,\I) \hookrightarrow \I^A$.  As we now want to think of this space as a measurable space in its own right we could emphasize it explicitly using the notation $(\C(A,\I),\sa)$.  However,  for our purposes the object $(\C(A,\I),\sa)$ will be used within the context of the category $\M$ so  the notation $\C(A,\I)$ is used to represent the measurable space.  Using this notation, the object $\C(\I,\I)$, within the context of being an object in $\M$, is the set of affine morphisms endowed with the $\sigma$-algebra making the evaluation maps \mbox{$\C(\I,\I) \stackrel{ev_a}{\longrightarrow} \I$} at every point in $a \in \I$ a measurable function.
 
  The space  $\I^{\C(A,\I)}$ is the set $\M( \C(A,\I), \I)$ endowed with the initial  $\sigma$-algebra making the evaluation maps $\I^{\C(A,\I)} \stackrel{ev_h}{\longrightarrow} \I$ measurable, for every $h \in \C(A,\I)$.  In this expression,  $\I^{\C(A,\I)}$, both the ``exponent'' $\C(A,\I)$ and the ``base'' $\I$ are measurable spaces so the standard construction for function spaces in $\M$ applies.    Let  $\C_w(\C(A,\I),\I)$ denote the subset of  $\I^{\C(A,\I)}$ consisting of the weakly averaging affine functions which preserve limits, and endow $\C_w(\C(A,\I),\I) \subseteq \I^{\C(A,\I)}$ with the subspace $\sigma$-algebra.  Consequently, for every   $h \in \C(A,\I)$,
all  the evaluation maps

\be \nonumber
 \begin{tikzpicture}[baseline=(current bounding box.center)]
 	\node	(prod) 	at	(0,0)              {$\C_w(\C(A,\I),\I)$};
	\node         (I)    at      (4,0)             {$\I$};
	\draw[->, above] (prod) to node  {$ev_h$} (I);
 \end{tikzpicture}
\ee
\noindent
are measurable.     For $k:A \rightarrow B$ an  affine morphism of convex spaces we have the map

\begin{equation}   \nonumber
 \begin{tikzpicture}[baseline=(current bounding box.center)]
         \node  (PA)  at (0,0)    {$\C_w(\C(A,\Rbar),\I)$};
         \node  (I)     at    (8,0)     {$ \C_w(\C(B,\Rbar),\I)$};

         \node  (PA2)  at (0,-.7)    {$G$};
         \node  (I2)     at    (8,-.7)     {$G \circ \C(k,\I)$};
         
	\draw[->,above] (PA) to node {$\C_w(\C(k,\Rbar),\I)$} (I);
	\draw[|->] (PA2) to node {} (I2);
	
\end{tikzpicture}
 \end{equation}
 \noindent
defined on all $g \in \C(B,\I)$ by $\C(k,\I)(g) = g \circ k$.    This yields a functor
\be \nonumber
\iota(\bullet) \stackrel{def}{=} \C_w(\C(\bullet,\Rbar),\I): \C \rightarrow \M.
\ee

 \begin{lemma} \label{cD} For $k: A \rightarrow B$ and $g: B \rightarrow \Rbar$ affine morphisms in $\C$  the $\M$ diagram
\begin{equation}   \nonumber
 \begin{tikzpicture}[baseline=(current bounding box.center)]
         \node   (ARR)  at  (0,1.75)    {$\Mc(A)$};
         \node  (BRR)  at (0,0)    {$\Mc(B)$};
         \node  (I)  at  (3.2,0)   {$\I$};

	\draw[->,above,right] (ARR) to node [yshift=3pt]{$ev_{g \circ k}$} (I);
	\draw[->,left] (ARR) to node {$\iota(k)$} (BRR);
	\draw[->,below] (BRR) to node {$ev_g$} (I);

	\end{tikzpicture}
 \end{equation}
 \noindent
 commutes.
 \end{lemma}
 \begin{proof}
 For all $K \in \Mc(A)$
 \be  \nonumber
 \begin{array}{lcl}
 ev_{g \circ k}(K) &=& K(g \circ k) \\
 &=& (K \circ \C(k,\I))(g) \\
 &=& ev_g(K \circ \C(k,\I)) \\
 &=& \left(ev_g \circ \Mc(k)\right)(K)
 \end{array}
 \ee
 \end{proof}
  
\section{The Giry monad as a codensity monad} 
The functor $\iota:\C \rightarrow \M$ induces the functor
\be \nonumber
\begin{array}{llccc}
\M^{\iota} &:& \M^{\M} & \longrightarrow & \M^{\C} \\
&:_{ob}& F & \mapsto & F \circ \iota \\
&:_{ar}& F \stackrel{\alpha}{\longrightarrow} G & \mapsto & F \circ \iota \stackrel{\alpha \circ \iota}{\longrightarrow} G \circ \iota
\end{array}
\ee
and a universal arrow from the functor $\M^{\iota}$ to the object $\iota \in_{ob} \M^{\C}$ is  called the right Kan extension of $\iota$ along $\iota$.  Like any universal arrow the right Kan extension is a pair $(\Ri, \epsilon)$ where $\Ri \in_{ob} \M^{\M}$ and $\epsilon: \Ri \circ \iota \Rightarrow \iota$ is the natural transformation having the property  such that
if $\alpha: \mS \circ \iota \Rightarrow \iota$ then  there exists a unique natural transformation $\overline{\alpha}: \mS \Rightarrow \Ri$ such that the diagram on the right in
\begin{figure}[H]
\begin{equation}   \nonumber
 \begin{tikzpicture}[baseline=(current bounding box.center)]
         \node  (P)   at  (-2.5,0)     {$\Ri$};
         \node  (S)  at (-2.5,-1.8)    {$\mS$};
         \node  (PM)  at (0,0)   {$\Ri \circ \iota$};
         \node  (SM) at (0,-1.8)  {$\mS \circ \iota$};
         \node  (M)  at  (2.6,0)    {$\Mc$};
         
         \node  (com) at  (-2,-2.7)  {in $\M^{\M}$};      
         \node  (com2) at  (1,-2.7)  {in $\M^{\C}$};      
                   
	\draw[->,left] (S) to node {$\overline{\alpha}$} (P);
	\draw[->,above] (PM) to node {$\epsilon$} (M);
	\draw[->,right,below] (SM) to node [xshift=2pt,yshift=-2pt] {$\alpha$} (M);
	\draw[->,left] (SM) to node {${\overline{\alpha}} \circ \iota$} (PM);
	
	 \end{tikzpicture}
 \end{equation} 
  \caption*{Diagram 3. The codensity monad of $\iota$ as a universal arrow.}
 \label{fig:rightKanU}
 \end{figure}
 \noindent
commutes.  If the pair $(\Ri, \epsilon)$ is the right Kan extension of $\iota$ along $\iota$ the $\Ri$ is a codensity monad of $\iota$.
The property of being a right Kan extension of $\iota$ along $\iota$ can equivalently be expressed in terms of the diagram    
\begin{figure}[H]
\begin{equation}   \nonumber
 \begin{tikzpicture}[baseline=(current bounding box.center)]
         \node  (C)   at  (0,0)     {$\C$};
         \node  (M1)  at (3.5,0)    {$\M$};
         \node  (M2)  at (3.5,-2.5)   {$\M$};
         \node  (comp) at  (2.8,-.5)  {$$};      
         \node  (comp2) at  (1.8,-.9)  {$$};      
         \node  (ph1)  at  (3.1,-1.3) {};
         \node  (ph2)  at  (3.8,-1.3) {};
                   
	\draw[->,above] (C) to node {$\iota$} (M1);
	\draw[->,left] (C) to node [xshift=-3pt] {$\iota$} (M2);
	\draw[->,out=235,in=135,looseness=.5,left] (M1) to node [xshift=2pt,yshift=0pt] {$\Ri$} (M2);
	\draw[->,out=-60,in=60,looseness=.5,right] (M1) to node [xshift=2pt,yshift=0pt] {$\mS$} (M2);
	
	\draw[->,above] (ph2) to node {$\overline{\alpha}$} (ph1);
	\draw[->,above] (comp) to node {$\epsilon$} (comp2);
		
	 \end{tikzpicture}
 \end{equation} 
 \end{figure}
 \noindent
which characterizes the property that the natural transformation $\epsilon: \Ri \circ \iota \Rightarrow \iota$ satisfies the condition that if  the pair $(\mS,\alpha)$  also yields a natural transformation $\alpha: \mS \circ \iota \Rightarrow \iota$ then there exists a unique natural transformation $\overline{\alpha}: \mS \Rightarrow \Ri$ such that $\alpha = \epsilon \circ (\overline{\alpha} \circ \iota)$.   

As $\M$ is complete the codensity monad $\Ri$ can be constructed pointwise \\
 \mbox{\cite[Theorem 1, page 233]{MacLane}} using the slice category $\Xit$ of objects under $X \in_{ob} \M$, which has  objects and arrows 
\begin{figure}[H]
\begin{equation}   \nonumber
 \begin{tikzpicture}[baseline=(current bounding box.center)]
         \node  (ob)  at  (-2,-1)  {objects: $(f,A)$};
         \node  (X)  at (0,0)    {$X$};
         \node  (iA)  at    (0,-2)     {$\iota(A)$};
         
	\draw[->,right] (X) to node {$f$} (iA);
	
	 \node  (ar)  at  (2.8,-1)  {arrows: $k$};
         \node  (X2)  at (5,0)    {$X$};
         \node  (iA2)  at    (4,-2)     {$\iota(A)$};
         \node  (iB2)  at    (6,-2)     {$\iota (B)$};
         
	\draw[->,left] (X2) to node {$f$} (iA2);
	\draw[->,right] (X2) to node {$g$} (iB2);
	\draw[->,below] (iA2) to node {$\iota(k)$} (iB2);
	
\end{tikzpicture}
 \end{equation}
   \caption*{Diagram 4.  The slice category $\Xit$.}
 \label{fig:slice}
 \end{figure}
 \noindent
where $A,B \in_{ob} \C$, $f,g \in_{ar} \M$ and $k \in_{ar} \C$.  There is a projection functor \mbox{$Q:\Xit \rightarrow \C$} mapping the objects $(f,A) \mapsto A$ and arrows $k \mapsto k$ which when composed with $\iota$ yields a composite functor whose limit
\be \nonumber    
 \displaystyle{ \lim_{\leftarrow}} \left( \Xit \stackrel{Q}{\longrightarrow} \C \stackrel{\iota}{\longrightarrow} \M \right) 
\ee
we claim is precisely $\T(X)$.  Towards this end we require the following construction.

Suppose \mbox{$f: X \rightarrow \iota(A)$} is an object in the slice category $\Xit$.    
Since $\iota(A)$ is a subobject $\iota(A) \hookrightarrow \I^{\C(A,\Rbar)}$ in $\M$,  by  the SMCC structure of $\M$,  $f$ determines a map
\begin{equation}   \label{hatDef}
 \begin{tikzpicture}[baseline=(current bounding box.center)]
         \node  (X)  at (0,0)    {$ \C(A,\Rbar)$};
         \node  (I)     at    (3,0)     {$\I^X$};
         \node  (com)  at  (8,0)  {$\hat{f}[h](x)=f(x)[h] \quad \forall h \in \C(A,\Rbar), \, \forall x \in X$};         
	\draw[->,above] (X) to node {$\hat{f}$} (I);
	
\end{tikzpicture}
 \end{equation}
 \noindent
where  the notation ``$[h]$'' is used to emphasize that the argument is itself a function and to avoid excessive parentheses.  Using  the definition of $\hat{f}$ and the fact $f(x) \in \iota(A)$ we obtain following result.
\begin{lemma} Given the object $f: X\rightarrow \iota(A)$ in $\Xit$ the map $\hat{f}$ defined by (\ref{hatDef}) is a weakly averaging affine map which preserves limits.
\end{lemma}

Given any $G \in \T(X)$ its composite with $\hat{f}$ gives the ``pushforward'' map
\begin{equation}   \nonumber
 \begin{tikzpicture}[baseline=(current bounding box.center)]
         
         \node  (AY)  at (0,2)    {$\C(A,\Rbar)$};
         \node  (AX)  at (0,.0)    {$\I^X$};
          \node  (A)  at  (2,0)      {$\I$};
          
	\draw[->,left] (AY) to node [xshift=-5pt] {$\hat{f}$} (AX);

	\draw[->, right]  (AY) to node [xshift=5pt]  {$G \circ \hat{f}$} (A);
	\draw[->, below] (AX) to node {$G$} (A);
	
	 \end{tikzpicture}
 \end{equation}
which is a weakly averaging affine map preserving limits because the components defining it are and hence $G \circ \hat{f} \in \iota(A)$.  To prove that the monad $(\T,\eta,\mu)$ yields the functor part of the right Kan extension of the $\iota$ along $\iota$ we show that it coincides with the point-wise construction given in \cite[Theorem 1, page 237]{MacLane}.

\begin{thm} (The right Kan extension of $\iota$ along $\iota$ as a point-wise limit.)  For each $X \in_{ob} \M$,   $\T(X) =  \displaystyle{ \lim_{\leftarrow}} \left( \Xit \stackrel{Q}{\longrightarrow} \C \stackrel{\iota}{\longrightarrow} \M \right) $ with the natural transformation $\lambda$ of the cone $(\T(X),\lambda)$ over $\iota \circ Q$  specified by
\begin{equation}   \nonumber
 \begin{tikzpicture}[baseline=(current bounding box.center)]
         \node  (PX)  at  (-3.5,1.3)  {$\T(X)$};
         \node  (iB)  at    (2.,1.3)   {$\iota(A)$};
         
         \node  (IX)  at  (-4.5,0)  {$\I^X$};
         \node  (I0)  at  (-3,0) {$\I$};
         
         \node  (X)  at (0.5,0)    {$ \C(A,\Rbar)$};
         \node  (I2)   at    (3.5,0)   {$\I$};
                  
         \draw[->,above] (IX) to node {$G$} (I0);
	\draw[->,above] (X) to node {$G \circ \hat{f}$} (I2);
	
	\draw[|->,densely dashed] (I0) to node {} (X);
	\draw[->,above] (PX) to node {$\lambda_{f}$} (iB);

\end{tikzpicture}
 \end{equation}
 \noindent
 for every $f \in \M(X, \iota(A))$, every $A \in_{ob} \C$ and where $\hat{f}$ is defined by (\ref{hatDef}).     Each $\theta \in \M(Y,X)$  induces a unique arrow 
 \be \nonumber
\T(\theta): \displaystyle{ \lim_{\leftarrow}}(\iota \circ Q) \longrightarrow \displaystyle{ \lim_{\leftarrow}}(\iota \circ Q')
\ee
commuting with the limiting cones, where $Q': (Y \! \! \downarrow \! \iota) \rightarrow \C$ is the projection functor.  By
defining    
\mbox{ $\epsilon: \T \circ \iota \rightarrow \iota$} componentwise by
 \be \nonumber
\epsilon_A = \lambda_{id_{\iota(A)}} \quad \textrm{ for all }A \in_{ob} \C
\ee
the pair $(\T,\epsilon)$ is the right Kan extension of $\iota$ along $\iota$.
\end{thm}
\begin{proof}  The proof is broken into multiple parts which are denoted using italicized headings.

\begin{flushleft} \emph{ $(\T(X),\lambda)$ as a limit cone over $\iota \circ Q$} \end{flushleft}
Given the $\Xit$-arrows
\begin{equation} \nonumber
\begin{tikzpicture}[baseline=(current bounding box.center)]

         \node  (X)  at  (-1,0)    {$X$};
         \node  (iA)     at    (2,1.3)     {$\iota(A)$};
         \node  (iB)     at    (2,-1.3)     {$\iota(B)$};
         \node  (A)  at  (6.,1.3)     {$A$};
         \node  (B)  at  (6,-1.3)     {$B$};
         
         \node (com1)  at  (1,-2.2)   {in $\Xit$};
         \node  (com2) at  (6,-2.2)  {in $\C$};
         
	\draw[->,above] (X) to node [xshift=-8pt] {$f$} (iA);
	\draw[->,below] (X) to node [xshift=-8pt] {$g$} (iB);
	\draw[->,right] (iA) to node {$\iota(k)$} (iB);
	\draw[->,right] (A) to node {$k$} (B);
	
\end{tikzpicture}
 \end{equation}
 \noindent
the commutativity condition $g = \iota(k) \circ f$ implies  $\hat{g} = \hat{f} \circ \C(k,\I)$. This  gives the commutative diagram on the left in the figure below, which in turn makes $(\T(X),\lambda)$  a cone over $\iota \circ Q$ 
because for every $G \in \T(X)$  the diagram on the right 
\begin{equation}   \nonumber
 \begin{tikzpicture}[baseline=(current bounding box.center)]
 
        \node (CBI)  at  (-7,1.3)  {$\C(B,\I)$};
        \node (CAI)  at   (-7,-1.2)    {$\C(A,\I)$};
        \node (IX1)  at  (-4.55,1.3)  {$\I^X$};
        \node (IX2) at   (-4.5,-1.2) {$\I^X$};
        \node (I) at (-2.9,0)   {$\I$};
        
        \draw[->,left] (CAI) to node {$\C(k,\I)$} (CBI);
        \draw[->,above] (CBI) to node {$\hat{f}$} (IX1);
        \draw[->,above] (CAI) to node {$\hat{g}$} (IX2);
        \draw[->,above] (IX1) to node {$G$} (I);
        \draw[->,below] (IX2) to node {$G$} (I);
 
         \node  (PA)  at (-.5,0)    {$I^X \stackrel{G}{\longrightarrow} \I$};
         
         \node  (ph1)     at    (4.5,1.3)     {};
         \node  (ph2)     at    (4.5,-.6)    {};
         
         \node (AR)  at   (2.5,1.3)  {$\C(A,\Rbar)$};
         \node (I1)    at    (5.5,1.3)  {$\I$};
         
          \node (BR)  at   (2.5,-1.2)  {$\C(B,\Rbar)$};
         \node (I2)    at    (5.5,-1.2)  {$\I$};
                  
	\draw[|->,above] (PA) to node [xshift=-8pt] {$\lambda_{f}$} (iA);
	\draw[|->,below] (PA) to node [xshift=-8pt] {$\lambda_{g}$} (iB);
	\draw[|->,right] (ph1) to node {$\iota(k)$} (ph2);
	
	\draw[->,above] (AR) to node {$G \circ \hat{f}$} (I1);
	\draw[->,above] (BR) to node {$G \circ \hat{g}$} (I2);
	
\end{tikzpicture}
 \end{equation}
commutes.

Now suppose that $(Z, \omega)$ is also a cone over the functor $\iota \circ Q$.  We must show there exist a unique arrow $\theta$ making the diagram
\begin{equation}   \nonumber
 \begin{tikzpicture}[baseline=(current bounding box.center)]
         \node  (Z)  at  (-3,0)   {$Z$};
         \node  (PA)  at (-1,0)    {$\T(X)$};
         \node  (iA)     at    (2,1.3)     {$\iota(A)$};
         \node  (iB)     at    (2,-1.3)     {$\iota(B)$};
         \node  (A)  at  (6.,1.3)     {$A$};
         \node  (B)  at  (6,-1.3)     {$B$};
         
         \node (com1)  at  (1,-2)   {in $\M$};
         \node  (com2) at  (6,-2)  {in $\C$};
         
	\draw[->,right] (PA) to node [xshift=3pt,yshift=-3pt] {$\lambda_{f}$} (iA);
	\draw[->,right] (PA) to node [xshift=3pt,yshift=3pt] {$\lambda_{g}$} (iB);
	\draw[->,right] (iA) to node {$\iota(k)$} (iB);
	\draw[->,right] (A) to node {$k$} (B);
	
	\draw[->,out=50,in=180,above] (Z) to node {$\omega_{f}$} (iA);
	\draw[->,out=-50,in=180,below] (Z) to node {$\omega_{g}$} (iB);	
	\draw[->,dashed,above] (Z) to node {$\theta$} (PA);
\end{tikzpicture}
 \end{equation}
 \noindent
commute.  The commutativity of the outer path implies  
 \be \label{omegaCondition}
 \omega_{g}(z)[h] = \omega_{f}(z)[h \circ k] \quad  \forall h \in \C(B,\Rbar), \, \forall z \in Z.
 \ee
 
 Fix  an element element $z \in Z$.  To  satisfy the required commutativity condition $\omega_f(z) = \lambda_f  \circ \theta(z)$ for all objects $f: X \rightarrow \iota(A)$ in $\Xit$ it is necessary and sufficient that the function $\theta(z)$ satisfy
the commutativity of the whole  $\M$ diagram (given $g = \iota(k) \circ f$)
\begin{equation}   \nonumber
 \begin{tikzpicture}[baseline=(current bounding box.center)]
 
        \node (CBI)  at  (-7,1.3)  {$\C(A,\I)$};
        \node (CAI)  at   (-7,-1.2)    {$\C(B,\I)$};
        \node (IX1)  at  (-4.,1.3)  {$\I^X$};
        \node (IX2) at   (-4.,-1.2) {$\I^X$};
        \node (I) at (-2.,0)   {$\I$};
        
        \draw[->,left] (CAI) to node {$\C(k,\I)$} (CBI);
        \draw[->,above] (CBI) to node {$\hat{f}$} (IX1);
        \draw[->,below] (CAI) to node {$\hat{g}$} (IX2);
        \draw[->,above] (IX1) to node [xshift=2pt]{$\theta(z)$} (I);
        \draw[->,below] (IX2) to node {$\theta(z)$} (I);
        \draw[->,below] (CBI) to node [xshift=-22,yshift=5] {$\omega_f(z)$} (I);
         \draw[->,above] (CAI) to node {$\omega_g(z)$} (I);
\end{tikzpicture}
 \end{equation}
and that $\theta(z) \in \T(X)$.

This commutativity condition can be used to define $\theta(z)$ because, for the convex space $\I$ \emph{which is also measurable}, every  $\gamma \in \I^X$ determines a measurable map $\I^{\I} \rightarrow \I^X$, and using the SMCC property of $\M$, we obtain the 
 object  $\gamma':X \rightarrow \iota(\I)$ in  $\Xit$ specified by 
\be  \label{primeConstruction}
\gamma'(x)[h] = h(\gamma(x)) \quad \forall h \in \C(\I,\I), \, \forall x \in X.
\ee
The map $\gamma'(x) \in \iota(\I)$ because the convex structure is defined pointwise on $\I^{\I}$ and the weakly averaging condition is clearly satisfied by the above definition.
This map in turn,  via the construction in (\ref{hatDef}), determines the map \mbox{$\hat{\gamma'}:\C(\I,\I) \rightarrow \I^X$}  in $\M$ specified by
\be  \nonumber
\hat{ \gamma'}[h](x) = \gamma'(x)[h] = h(\gamma(x)) \quad  \forall h \in \C(\I,\I), \, \forall x \in X.
\ee
This map $\hat{\gamma'}$ is also a weakly averaging affine function because $h$ is affine and $\gamma(x)$ is weakly averaging for all $x \in X$.

Observe that for $\gamma \in \M(X,\I)$ it follows 
 $\hat{ \gamma'}[id_{\I}](x) = \gamma'(x)[id_{\I}] = \gamma(x)$  for all $x \in X$ and consequently $\gamma = \hat{ \gamma'}[id_{\I}]$.  Using this property  the required commutativity  condition $\omega_f(z) = \theta(z) \circ \hat{f}$ is used  to define $\theta(z)$ as
\be \label{defEta2}
\theta(z)[\gamma] = \theta(z)\left[\hat{\gamma'}[id_{\I}]\right] \stackrel{def}{=} \omega_{\gamma'}(z)[id_{\I}] \quad \forall z \in Z, \, \forall \gamma \in \M(X,\I)
\ee
or equivalently, for every $z \in Z$ the map $\theta(z)$ is defined \emph{as a function} by the commutativity of the diagram
\begin{equation}   \nonumber
 \begin{tikzpicture}[baseline=(current bounding box.center)]
         \node  (AY)  at (0,2)    {$\C(\I,\Rbar)$};
         \node  (AX)  at (0,.0)    {$\I^X$};
          \node  (A)  at  (2,0)      {$\I$};
          
	\draw[->,left] (AY) to node [xshift=-5pt] {$\hat{\gamma'}$} (AX);

	\draw[->, right]  (AY) to node [xshift=5pt]  {$\theta(z) \circ \hat{\gamma'} \stackrel{def}{=} \omega_{\gamma'}(z)$} (A);
	\draw[->, below] (AX) to node {$\theta(z)$} (A);
	
	 \end{tikzpicture}
 \end{equation}
for every $\gamma \in \M(X,\I)$.  

We now proceed to verify the condition that $\theta(z) \in \T(X)$ by proving  $\theta(z)$  is weakly averaging,  affine, and preserves limits.   

\begin{flushleft}\emph{Weakly averaging condition}\end{flushleft}  Let  $\overline{u} \in \M(X,\I)$ be a constant function with value $u \in \I$. Let $k \in \C(\I,\I)$ be the constant map $k = \overline{u}$ where we retain the  symbol ``$k$'' to avoid confusion between the two constant functions with value $u$.   For any $\gamma \in \M(X,\I)$  the diagram on the left in
 \begin{figure}[H]
\begin{equation}   \nonumber
 \begin{tikzpicture}[baseline=(current bounding box.center)]
         \node  (X)  at (0,0)    {$X$};
         \node  (Iu)  at (2,1)    {$\iota(\I)$};
         \node  (Iv)  at  (2,-1)      {$\iota(\I)$};
         \node (com1) at (1,-2)   {in $\Xit$};
          
	\draw[->,above] (X) to node [xshift=0pt] {$\gamma'$} (Iu);
	\draw[->, below]  (X) to node [xshift=0pt]  {$\overline{u}'$} (Iv);
	\draw[->, right] (Iu) to node {$\iota(k)$} (Iv);

         \node  (Z)  at (4,0)    {$Z$};
         \node  (I1)  at (6,1)    {$\iota(\I)$};
         \node  (I2)  at  (6,-1)      {$\iota(\I)$};
         \node  (com2) at  (5,-2)  {in $\M$};
          
	\draw[->,above] (Z) to node [xshift=0pt] {$\omega_{\gamma'}$} (I1);
	\draw[->, below]  (Z) to node [xshift=0pt]  {$\omega_{\overline{u}'}$} (I2);
	\draw[->, right] (I1) to node {$\iota(k)$} (I2);
		
	 \end{tikzpicture}
 \end{equation} 
     \caption*{Diagram 5.  Determining the values $\omega_{\overline{u}'}$.}
 \label{fig:constant}
  \end{figure}
 \noindent
commutes because, for all $x\in X$,  the composite map $\iota(k)(\gamma'(x))$ is the pushforward map
\begin{equation}   \nonumber
 \begin{tikzpicture}[baseline=(current bounding box.center)]
         \node  (C1)  at (0,0)    {$\C(\I,\I)$};
         \node  (C2)  at (0,.-2)    {$\C(\I,\I)$};
          \node  (I)  at  (2.5,0)      {$\I$};
          
	\draw[->,above] (C1) to node [xshift=0pt] {$\gamma'(x)$} (I);

	\draw[->, left]  (C2) to node [xshift=0pt]  {$\C(k,\I)$} (C1);
	\draw[->, right] (C2) to node {$\overline{u}'(x)$} (I);
	
         \node  (C12)  at (5,0)    {$h \circ k$};
         \node  (C22)  at (5,.-2)    {$h$};
          \node  (I2)  at  (9.5,0)      {$\gamma'(x)(h \circ k) = h(k(\gamma(x)) = h(u)$};
          
	\draw[|->,above] (C12) to node [xshift=0pt] {$$} (I2);

	\draw[|->, left]  (C22) to node [xshift=0pt]  {$$} (C12);
	\draw[|->, right] (C22) to node {$$} (I2);

	 \end{tikzpicture}.
 \end{equation}
Given the cone $(Z,\omega)$ over the functor $\iota \circ Q$ it follows that the diagram on the right in \mbox{Diagram 5} also commutes which entails that
\be \nonumber
\begin{array}{lcl}
\omega_{\overline{u}'}(z)[id_{\I}] &=&  \left(\iota(k) \circ \omega_{\gamma'}\right)(z)[id_{\I}]  \\
&=&  \omega_{\gamma'}(z)[k]  \\
&=&   \omega_{\gamma'}(z)[\overline{u}] \\
&=& u
\end{array}
\ee
where the last equality follows because  $\omega_{\gamma'}(z)$ is weakly averaging.  This shows that
\be  \nonumber
\theta(z)[\overline{u}] = \omega_{\overline{u}'}(z)[id_{\I}]=u
\ee
which proves $\theta(z)$ is weakly averaging.
  
\begin{flushleft} \emph{Affine condition}\end{flushleft} Consider the object $\I \times \I$ in $\C$ which has a convex structure defined componentwise by
\be \nonumber
\sum_{i=1}^n r_i (u_i,v_i) = (\sum_{i=1}^n r_i u_i, \sum_{i=1}^n r_i v_i) \quad \textrm{where }\sum_{i=1}^n r_i = 1, \, \, \forall r_i \in \I.
\ee
Let $\alpha \in \I$.  The map  $\pi_1 +_{\alpha} \pi_2 : \I \times \I \rightarrow \I$ defined by $\pi_1 +_{\alpha} \pi_2:(u,v) \mapsto u+_{\alpha}v$  is  affine because
\be \nonumber
\begin{array}{lcl}
\left( \pi_1 +_{\alpha} \pi_2\right)  \left( \sum_{i=1}^n r_i (u_i,v_i) \right) &=&  \alpha \sum_{i=1}^n r_i u_i + (1-\alpha) \sum_{i=1}^n r_i v_i  \\ 
&=&    \sum_{i=1}^n r_i (\alpha u_i) +  \sum_{i=1}^n r_i \left( (1-\alpha) v_i\right)  \\
&=&  \sum_{i=1}^n r_i (u_i +_{\alpha} v_i) \\
&=& \sum_{i=1}^n r_i \left( (\pi_1 +_{\alpha} \pi_2) (u_i, v_i)\right)
\end{array}
\ee
Moreover this map $\pi_1 +_{\alpha} \pi_2$ is also measurable with $\I \times \I$ having the product $\sigma$-algebra.

For $\gamma_1,\gamma_2 \in \M(X,\I)$ we obtain the induced maps $\gamma_1', \gamma_2' \in \M(X,\iota(\I))$ by the construction given in (\ref{primeConstruction}) and  the  diagram on the left in
 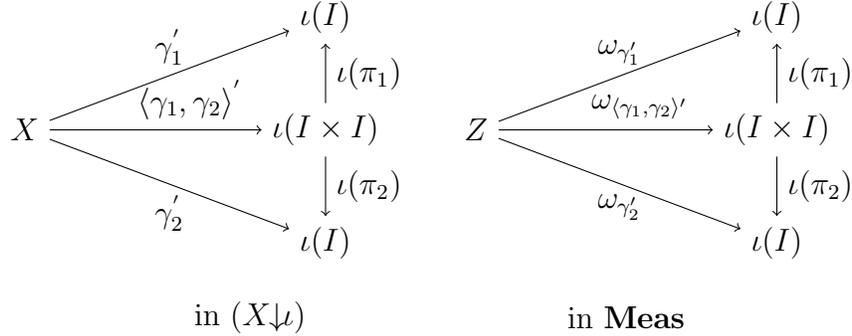
\begin{figure}[H]
\begin{equation}   \nonumber
 \begin{tikzpicture}[baseline=(current bounding box.center)]
         \node  (X)  at (-2,0)    {$X$};
         \node  (I1)  at (2,1.5)    {$\iota(\I)$};
         \node  (I2)  at  (2,-1.5)      {$\iota(\I)$};
         \node   (II)  at  (2,0)     {$\iota(\I \times \I)$};
         \node (com1) at (1,-2.5)   {in $\Xit$};
          
	\draw[->,above] (X) to node [xshift=0pt] {$\gamma_1^{'}$} (I1);
	\draw[->, below]  (X) to node [xshift=0pt]  {$\gamma_2^{'}$} (I2);
	\draw[->, above]  (X) to node [xshift=13pt]  {${\langle \gamma_1,\gamma_2 \rangle}^{'}$} (II);
	\draw[->, right] (II) to node {$\iota(\pi_1)$} (I1);
	\draw[->, right] (II) to node {$\iota(\pi_2)$} (I2);

         \node  (Z)  at (4,0)    {$Z$};
         \node  (I12)  at (8,1.5)    {$\iota(\I)$};
         \node  (I22)  at  (8,-1.5)      {$\iota(\I)$};
         \node  (II2)  at   (8,0)    {$\iota(\I \times \I)$};
         \node  (com2) at  (6,-2.5)  {in $\M$};
          
	\draw[->,above] (Z) to node [xshift=0pt] {$\omega_{\gamma_1'}$} (I12);
	\draw[->, below]  (Z) to node [xshift=0pt]  {$\omega_{\gamma_2'}$} (I22);
	\draw[->,above] (Z) to node [xshift=13pt] {$\omega_{\langle \gamma_1,\gamma_2 \rangle'}$} (II2);
	\draw[->, right] (II2) to node {$\iota(\pi_1)$} (I12);
	\draw[->, right] (II2) to node {$\iota(\pi_2)$} (I22);
		
	 \end{tikzpicture}
 \end{equation} 
     \caption*{Diagram 6.  Determining the affine property.}
 \label{fig:aff}
  \end{figure}
 \noindent
commutes.  Hence, given the cone $(Z, \omega)$ over the functor $\iota \circ Q$, the diagram on the right also commutes,  yielding the two equations
\be \nonumber
\omega_{\langle \gamma_1,\gamma_2 \rangle'}(z)[\pi_1] = \omega_{\gamma_1'}(z)[id_{\I}] \quad \textrm{and} \quad \omega_{\langle \gamma_1,\gamma_2 \rangle'}(z)[\pi_2] = \omega_{\gamma_2'}(z)[id_{\I}]
\ee
for all $z \in Z$.
As $\omega_{\langle \gamma_1, \gamma_2 \rangle'}(z)$ is affine it follows, using the two above equations, that
\be \label{affineproperty}
\omega_{\langle \gamma_1,\gamma_2 \rangle'}(z)[\pi_1 +_{\alpha} \pi_2] = \omega_{\gamma_1'}(z)[id_{\I}]  +_{\alpha} \omega_{\gamma_2'}(z)[id_{\I}]  \quad \forall z \in Z
\ee

Now using the commutative diagram on the left in 
\begin{equation}   \nonumber
 \begin{tikzpicture}[baseline=(current bounding box.center)]
         \node  (X)  at (-1,0)    {$X$};
         \node  (II)  at (2,0)    {$\iota(\I \times \I)$};
         \node  (I)  at  (2,-1.5)      {$\iota(\I)$};
         \node (com1) at (1,-2.5)   {in $\Xit$};
          
	\draw[->, above]  (X) to node [xshift=9pt]  {${\langle \gamma_1,\gamma_2 \rangle}^{'}$} (II);
	\draw[->, right] (II) to node {$\iota(\pi_1 +_{\alpha} \pi_2)$} (I);
	\draw[->, left] (X) to node [yshift=-5pt]{$(\gamma_1 +_{\alpha} \gamma_2)'$} (I);

         \node  (Z)  at (5,0)    {$Z$};
         \node  (II2)  at   (8,0)    {$\iota(\I \times \I)$};
         \node  (I2)  at  (8,-1.5)      {$\iota(\I)$};
         \node  (com2) at  (6,-2.5)  {in $\M$};
          
	\draw[->, below]  (Z) to node [xshift=-13pt]  {$\omega_{(\gamma_1 +_{\alpha} \gamma_2)'}$} (I2);
	\draw[->,above] (Z) to node [xshift=4pt] {$\omega_{\langle \gamma_1,\gamma_2 \rangle'}$} (II2);
	\draw[->, right] (II2) to node [xshift=0pt]{$\iota(\pi_1 +_{\alpha} \pi_2)$} (I2);
		
	 \end{tikzpicture}
 \end{equation} 
 \noindent
it follows that the diagram on the right also must commute which gives the equation
\be \label{additivity}
\omega_{(\gamma_1 +_{\alpha} \gamma_2)'}(z)[id_{\I}] = \omega_{\langle \gamma_1,\gamma_2 \rangle'}(z)[\pi_1 +_{\alpha} \pi_2].
\ee
Combining the last two results and using the definition of $\theta(z)$ it follows that
\be \nonumber
\begin{array}{lcll}
\theta(z)(\gamma_1 +_{\alpha} \gamma_2) &=& \omega_{(\gamma_1 +_{\alpha} \gamma_2)'}(z)[id_{\I}]  & \textrm{ by def. of }\theta(z)(\gamma_1+_{\alpha} \gamma_2) \\
&=& \omega_{\langle \gamma_1,\gamma_2 \rangle'}(z)[\pi_1 +_{\alpha} \pi_2]  & \textrm{ by }(\ref{additivity})\\
&=&  \omega_{\gamma_1'}(z)[id_{\I}]  +_{\alpha} \omega_{\gamma_2'}(z)[id_{\I}] & \textrm{ by } (\ref{affineproperty}) \\
&=& \theta(z)(\gamma_1) +_{\alpha} \theta(z)(\gamma_2) & \textrm{ by def. of }\theta(z)(\gamma_1) \, \textrm{and} \, \theta(z)(\gamma_1)
\end{array}
\ee
which shows that $\theta(z)$ is affine for all $z \in Z$.

\begin{flushleft}\emph{Preserves Limits} \end{flushleft}
Consider the commutative $\M$-diagram on the left 
\begin{equation}   \nonumber
 \begin{tikzpicture}[baseline=(current bounding box.center)]
         \node  (X)  at  (-1,0)   {$X$};
         \node  (TX)  at (2,-1)    {$\T(X)$};
         \node  (I)     at    (2,1)     {$\I$};
         \node (com1) at (0,-2)  {in $\M$};
          
	\draw[->,left] (X) to node [xshift=3pt,yshift=5pt] {$f$} (I);
	\draw[->,below] (X) to node [xshift=3pt,yshift=0pt] {$\eta_X$} (TX);
	\draw[->,right] (TX) to node {$ev_f $} (I);
	
	 \node  (X2)  at  (6,0)   {$X$};
         \node  (TX2)  at (9,-1)    {$\iota(\T(X))$};
         \node  (I2)     at    (9,1)     {$\iota(\I)$};
         \node  (com2)  at  (8,-2)  {in $\Xit$};
          
	\draw[->,left] (X2) to node [xshift=3pt,yshift=5pt] {$f'$} (I2);
	\draw[->,below] (X2) to node [xshift=3pt,yshift=0pt] {${\eta_X}'$} (TX2);
	\draw[->,right] (TX2) to node {$\iota(ev_f)$} (I2);

\end{tikzpicture}
 \end{equation}
where $\eta_X$ is the unit of the monad $\T$ mapping $x \mapsto ev_x$ (corresponding to the Dirac measure at $x$ when viewed in $\G(X)$).  Since  $ev_f$ is an affine mapping between the two convex spaces $\T(X)$ and $\I$ this gives the induced commutative diagram on the right in the category $\Xit$.  Thus we obtain the commutative $\Xit$-diagram
\begin{equation}   \nonumber
 \begin{tikzpicture}[baseline=(current bounding box.center)]
         \node  (Z)  at  (-4,0)   {$Z$};
         \node  (X)  at  (-2,0)   {$\T(X)$};
         \node  (TX)  at (1,-1)    {$\iota(\T(X))$};
         \node  (I)     at    (1,1)     {$\iota(\I)$};
          
	\draw[->,left] (X) to node [xshift=3pt,yshift=5pt] {$\lambda_{f'}$} (I);
	\draw[->,below] (X) to node [xshift=-3pt,yshift=1pt] {$\lambda_{{\eta_X}'}$} (TX);
	\draw[->,right] (TX) to node {$\iota(ev_f)$} (I);
	\draw[->,above] (Z) to node {$\theta$} (X);
	\draw[->,above, out=60,in=180,looseness=.5] (Z) to node {$\omega_{f'}$} (I);
	\draw[->,below, out=-60,in=180,looseness=.5] (Z) to node {$\omega_{{\eta_X}'}$} (TX);
\end{tikzpicture}
\end{equation}

This shows that for every $f \in \M(X,\I)$ the condition
\be \nonumber
\omega_{f'}(z) = \omega_{{\eta_X}'}(z)[ev_f]
\ee
holds.  For any  sequence of simple functions  \mbox{$\{f_i\}_{i=1}^{\infty} \rightarrow f$}  it follows, for all $G \in \T(X)$  the condition $\{ev_{f_i}(G) \}_{i=1}^{\infty} \rightarrow ev_f(G)$ holds by the monotone convergence theorem so $\{ev_{f_i} \}_{i=1}^{\infty} \rightarrow ev_f$.   But each function $ev_{f_i}:\T(X) \rightarrow \I$ is also a measurable function so is the limit of a sequence of simple functions $\{g_{i,j}\}_{j=1}^{\infty} \rightarrow ev_{f_i}$.  Consequently we can obtain a sequence of simple measurable functions $\{g_k\}_{k=1}^{\infty} \rightarrow ev_f$.
Since \mbox{$\theta(z)[f_i] \stackrel{\triangle}{=} \omega_{f_i'}(z)[id_{\I}]=  \omega_{\eta_X'(z)}[ev_{f_i}]$} we obtain
\be \nonumber
\lim \{ \theta(z)[f_i] \}_{i=1}^{\infty}  = \lim \{ \omega_{\eta_X'}(z)[g_k]\}_{k=1}^{\infty}   \rightarrow \omega_{\eta_X'}(z)[ev_f]=\theta(z)[f]
\ee
where the convergence follows from the fact $\omega_{\eta_{X}'}(z) \in \iota(\T(X))$ and hence preserves limits.

\begin{flushleft} \emph{Functoriality of $\T$}\end{flushleft}
The object $X$ with all the $\Xit$ arrows gives a cone over the functor $\iota \circ Q$ and the unique arrow $\eta: X \rightarrow \T(X)$ making the diagram
\begin{figure}[H]
\begin{equation}   \nonumber
 \begin{tikzpicture}[baseline=(current bounding box.center)]
         \node  (Z)  at  (-3,0)   {$X$};
         \node  (PA)  at (-1,0)    {$\T(X)$};
         \node  (iA)     at    (2,1.3)     {$\iota(A)$};
         \node  (iB)     at    (2,-1.3)     {$\iota(B)$};
         \node  (A)  at  (6.,1.3)     {$A$};
         \node  (B)  at  (6,-1.3)     {$B$};
         
         \node (com1)  at  (1,-2)   {in $\M$};
         \node  (com2) at  (6,-2)  {in $\C$};
         
	\draw[->,right] (PA) to node [xshift=3pt,yshift=-3pt] {$\lambda_{f}$} (iA);
	\draw[->,right] (PA) to node [xshift=3pt,yshift=3pt] {$\lambda_{g}$} (iB);
	\draw[->,right] (iA) to node {$\iota(k)$} (iB);
	\draw[->,right] (A) to node {$k$} (B);
	
	\draw[->,out=50,in=180,above] (Z) to node {$f$} (iA);
	\draw[->,out=-50,in=180,below] (Z) to node {$g$} (iB);	
	\draw[->,dashed,above] (Z) to node {$\eta$} (PA);
\end{tikzpicture}
 \end{equation}
    \caption*{Diagram 7.  The unit of the monad $\T$ at component $X$.}
 \label{fig:unit}
 \end{figure}
 \noindent

 \noindent
commute is precisely the unit of the monad $\T$ at the component $X$, $\eta=\eta_X$, because
\be \nonumber
\begin{array}{lcl}
(\lambda_{g} \circ \eta_X)(x) &=& ev_x \circ \hat{g} \\
&=&  g(x) \\
&=& f(x) \circ k  \quad \textrm{because  } g=\iota(k) \circ f\\
&=&  ev_x \circ \hat{f} \circ k \\
&=&(\iota(k) \circ \lambda_{f} \circ \eta_X)(x) 
\end{array}.
\ee
Consequently, by precomposition of the cone with vertex $X$ shown in Diagram 7, each $\theta \in \M(Y,X)$ induces a cone with vertex $Y$ over $\iota \circ Q$ and hence uniquely determines an arrow 
\be \nonumber
\T(\theta): \displaystyle{ \lim_{\leftarrow}}(\iota \circ Q) \longrightarrow \displaystyle{ \lim_{\leftarrow}}(\iota \circ Q')
\ee
where $Q': (Y \! \! \downarrow \! \iota) \rightarrow \C$ is the projection functor, making $\T$ functorial in the above construction which coincides with the previously defined operation of $\T$ on $\M$ arrows, i.e., as the pushforward map.

Having established that for each measurable space $X$ the pair $(\T(X), \{\lambda_g\}_{g \in_{ob} \Xit})$ forms a limiting cone over the functor $\iota \circ Q$ the rest of the proof  now follows the proof  of \cite[Theorem  1, p233]{MacLane} verbatim for the general construction of the pointwise right Kan extension of $\iota$ along $\iota$.  We give the proof showing the naturality of $\epsilon$, which expands upon the proof given by MacLane, and refer the reader to MacLanes proof  that if \mbox{$\mS: \M \rightarrow \M$} is another functor with $\alpha:\mS \circ \iota \rightarrow \iota$ a natural transformation then it corresponds bijectively with a natural transformation $\overline{\alpha}: \mS \rightarrow \T$.  This result simply depends upon $(\T(X), \{\lambda_{g}\}_{g \in_{ob} \Xit})$  being a limiting cone over the functor $\iota \circ Q$ and the functoriality of $\T$. 

\begin{flushleft} \emph{Defining the universal arrow $\epsilon$} \end{flushleft}
Let $A$ be an object of $\C$.   The identity map \mbox{$id_{\iota(A)}: \iota(A) \rightarrow \iota(A)$} is an object in the slice category $(\iota(A) \! \! \downarrow \! \! \iota)$ and, just as we defined in (\ref{hatDef}),  using  the SMCC structure of $\M$ we obtain a measurable map
 \be \nonumber
 \begin{tikzpicture}[baseline=(current bounding box.center)]
	
	 \node	(CAI) 	at	(0,0)              {$\C(A,\I)$};
	\node	(IiA)	at	(7.4,0)	               {$\I^{\iota(A)}$};
	
	\node        (A)   at     (-1,-1.3)       {$A$};
	\node        (I1)  at    ( 1,-1.3)       {$\I$};
	
	\node        (C)  at    ( 5, -1.3)      {$\iota(A)$};
	\node	(I2)	at	(10,-1.3)     {$\I$};

	\draw[->, above] (CAI) to node  {$\hat{id}_{\iota(A)}$} (IiA);
	\draw[->,above] (A) to node {$h$} (I1);
	
	\draw[|->,densely dashed] (I1) to node {} (C);
	\draw[->,above] (C) to node {$\hat{id}_{\iota(A)}(h)$} (I2);
	
	\node (K)    at  (5,-2.3)   {$K$};
	\node  (Kh) at  (10,-2.3) {$K(h)$};
	\draw[|->] (K) to node {} (Kh);

 \end{tikzpicture}
\ee
from which we see $\hat{id}_{\iota(A)}(h) = ev_h$, the evaluation map at $h$.   Corresponding to this object in the slice category $(\iota(A) \! \! \downarrow \! \! \iota)$ there is 
the component map $\lambda_{id_{\iota(A)}}:\T(\iota(A)) \rightarrow \iota(A)$ of the natural transformation $\lambda$, of the universal cone $(\T(\iota(A)),\lambda)$ over the functor $\iota \circ Q''$, where $Q'': (\iota(A) \! \! \downarrow \! \! \iota) \rightarrow \C$ is the projection functor.  
 The universal arrow $\epsilon$  is defined  componentwise at $A$ by 
\be \nonumber
 \begin{tikzpicture}[baseline=(current bounding box.center)]
	
	 \node	(PiA) 	at	(0,0)              {$\T(\iota(A))$};
	\node	(iA)	at	(7.4,0)	               {$\iota(A)$};
	\node        (IiA)   at     (-1,-.9)       {$\I^{\iota(A)}$};
	\node        (I1)  at    ( 1,-.9)       {$\I$};
	\node        (C)  at    ( 5, -.9)      {$\C(A,\I)$};
	\node         (ICiA) at   (7.9,-.9)    {$\I^{\iota(A)}$};
	\node	(I2)	at	(10,-.9)     {$\I$};

	\draw[->, above] (PiA) to node  {$\epsilon_A = \lambda_{id_{\iota(A)}}$} (iA);
	\draw[->,above] (IiA) to node {$G$} (I1);
	\draw[|->,densely dashed] (I1) to node {} (C);
	\draw[->,above] (C) to node {$\hat{id}_{\iota(A)}$} (ICiA);
	\draw[->,above] (ICiA) to node {$G$} (I2);

 \end{tikzpicture}
\ee
hence for all $G \in \T(\iota(A))$ the map $\epsilon_A(G)$ is specified  by $\epsilon_A(G)[h] = G(ev_h)$ for all \mbox{$h \in \C(A,\I)$}, and $\epsilon_A$  is a measurable weakly averaging affine map because both component maps are measurable weakly averaging affine maps.

\begin{flushleft} \emph{Naturality of $\epsilon$} \end{flushleft}
 For $k:A \rightarrow B$  a $\C$ morphism the naturality of $\epsilon$ requires the $\M$ diagram
\begin{figure}[H]
\begin{equation}   \nonumber
 \begin{tikzpicture}[baseline=(current bounding box.center)]
          
         \node  (PX)  at (-1.0,.0)    {$\T(\iota(A))$};
         \node  (X)  at  (3,.0)   {$\iota(A)$};
         \node  (PY)  at (-1.0,-1.8)    {$\T(\iota(B))$};
         \node (Y)  at (3,-1.8)   {$\iota(B)$};

	\draw[->,left] (PX) to node [yshift=5pt]{$\T(\iota(k))$} (PY);
	\draw[->,above] (PX) to node {$\epsilon_A$} (X);
	\draw[->,above] (PY) to node {$\epsilon_B$} (Y);
	\draw[->, right]  (X) to node {$\iota(k)$} (Y);

         \node  (PX2)  at (-1,.-3)    {$G$};
         \node  (X2)  at  (7.7,.-3)   {$\epsilon_A(G)$};
         \node  (PY2)  at  (-1,-4.8)   {$G \circ \C(\iota(k),\I)$};
          \node (Y2)  at (6,-4.8)   {$\epsilon_B(G \circ \C(\iota(k),\I))=\epsilon_A(G)\circ \C(k,\I)$};
          \node (ph)  at  (7.7,-4.7)  {};

	\draw[|->,above] (PX2) to node {$$} (PY2);
	\draw[|->,left] (PX2) to node {$$} (X2);
	\draw[|->,left] (PY2) to node {$$} (Y2);
	\draw[|->, above]  (X2) to node {}  (ph);

\end{tikzpicture}
 \end{equation}
  \caption*{Diagram 8.  Requirements for the naturality of $\epsilon$.}
 \label{fig:barycenter}
 \end{figure}
 \noindent
 to commute.
 Evaluating the  expression at the bottom right in the diagram at the affine morphism  $h: B \rightarrow \Rbar$ gives
 \be  \nonumber
 \begin{array}{lcl}
 \left(\epsilon_A(G)\circ \C(k ,\I)\right)[h] &=& \epsilon_A(G)[h \circ k] \\
 &=& G(ev_{h \circ k})  \\
 &=&  G(ev_h \circ \C(k,\I)) \quad \textrm{by Lemma~\ref{cD}} \\
 &=&  (G \circ \C(\iota(k),\I))(ev_h)  \\
 &=& \left( \lambda_{id_{\iota(B)}}(G \circ \C(\iota(k),\I))\right)[h] \\
 &=& \epsilon_B(\T(\iota(k))[G] )[h] \\
 &=& \epsilon_B(G \circ \C(\iota(k),\I))[h]
 \end{array}
 \ee
and hence $\epsilon$ is a natural transformation.
\end{proof}
To see that the monad $(\T, \eta, \mu)$ is a codensity monad of $\iota$ it suffices to show that the multiplication $\mu$, which defines a cone over $\iota \circ Q$ in the above proof, 
\begin{equation}   \nonumber
 \begin{tikzpicture}[baseline=(current bounding box.center)]
         \node  (Z)  at  (-4,0)   {$\T(\T(X))$};
         \node  (PA)  at (-1,0)    {$\T(X)$};
         \node  (iA)     at    (2,1.3)     {$\iota(A)$};
         \node  (iB)     at    (2,-1.3)     {$\iota(B)$};
         \node  (A)  at  (6.,1.3)     {$A$};
         \node  (B)  at  (6,-1.3)     {$B$};
         
         \node (com1)  at  (1,-2.5)   {in $\M$};
         \node  (com2) at  (6,-2.5)  {in $\C$};
         
	\draw[->,right] (PA) to node [xshift=3pt,yshift=-3pt] {$\lambda_{f}$} (iA);
	\draw[->,right] (PA) to node [xshift=3pt,yshift=3pt] {$\lambda_{g}$} (iB);
	\draw[->,right] (iA) to node {$\iota(k)$} (iB);
	\draw[->,right] (A) to node {$k$} (B);
	
	\draw[->,out=30,in=180,above] (Z) to node {$\mu_X(\_) \circ \hat{f}$} (iA);
	\draw[->,out=-30,in=180,below] (Z) to node {$\mu_X(\_) \circ \hat{g}$} (iB);	
	\draw[->,dashed,above] (Z) to node {$\theta$} (PA);
\end{tikzpicture}
 \end{equation}
and the resulting (unique) map $\theta$ is in fact $\theta = \mu_X$.  The proof is straightforward.

 \begin{flushleft}
Kirk Sturtz \\
Universal Mathematics \\
Dayton, OH USA \\
\email{kirksturtz@UniversalMath.com} \\
\end{flushleft}


\begin{thebibliography}{99}
\bibitem [Avery, 2014]{Avery} T. Avery, Codensity and the Giry monad.  \url{http://arxiv.org/abs/1410.4432}, 2014.


\bibitem[Doberkat, 2004]{Doberkat} E.E. Doberkat,  Eilenberg-Moore algebras for stochastic relations. Information and Computation, 204 (2006), 1756-1781.

\bibitem [Fritz, 2009]{Fritz} T. Fritz, Convex Spaces I: Definitions and Examples.  \url{http://arxiv.org/abs/0903.5522}, 2009.

\bibitem [Giry, 1982]{Giry} M. Giry, A categorical approach to probability theory, in Categorical Aspects of Topology and Analysis, Vol. 915,  68-85, Springer-Verlag, 1982. 

\bibitem [Kock, 1972]{Kock} A. Kock, Strong Functors and Monoidal Monads, Archiv der Math. 23 (1972), 113-120.

\bibitem [Lawvere and Rosebrugh, 2005]{LR} F.W. Lawvere and R. Rosebrugh, \textit{Sets for Mathematics}, Cambridge University Press, 2005.

\bibitem  [Leinster, 2013]{Leinster} T. Leinster, Codensity and the ultrafilter monad. Theory and Applications of Categories, Vol. 28, No. 13,  332 - 370 (2013).


\bibitem[MacLane, 1971]{MacLane} S. MacLane, Categories for the Working Mathematician, Springer-Verlag, 1971. 

\bibitem  [Meng, 1987]{Meng} X. Q. Meng, Categories of Convex Sets and of Metric Spaces, with Applications to Stochastic Programming.  Dissertation, SUNY Buffalo, 1987. 

\end{thebibliography}
 \end{document}